\documentclass[a4paper,12pt,reqno]{amsart}
\usepackage{latexsym,amsmath,url,caption2,epsfig}
\usepackage{textcomp}
\usepackage{amsfonts,euscript}
\usepackage{psfig}
\usepackage{graphicx}
\usepackage{amsmath}
\usepackage{amssymb}
\usepackage{amsfonts}
\usepackage{amsthm}
\usepackage{mathrsfs}
\usepackage[all]{xy}
\usepackage{epstopdf}
\usepackage{subfigure}
\usepackage{rotating}
\usepackage{appendix}
\usepackage{multirow,rotating}
\usepackage{url}
\usepackage{subfigure}
\usepackage{setspace}
\usepackage{listings}
\usepackage{color}
\usepackage[a4paper]{geometry}
\usepackage{graphics}
\usepackage[english]{babel}
\usepackage{datetime}
\usepackage{subfigure}
\usepackage{booktabs}
\usepackage{rotfloat}
\usepackage[numbers]{natbib}
\usepackage{algorithm}
\usepackage{algorithmicx}
\usepackage[hang,flushmargin]{footmisc}
\usepackage{lipsum}
\usepackage[noend]{algpseudocode}
\usepackage{framed}
\usepackage{lineno}

\makeatletter
\newcommand{\algorithmfootnote}[2][\footnotesize]{  \let\old@algocf@finish\@algocf@finish  \def\@algocf@finish{\old@algocf@finish    \leavevmode\rlap{\begin{minipage}{\linewidth}
				#1#2
\end{minipage}}  }}
\makeatother
\graphicspath{{pic/}}

\providecommand{\U}[1]{\protect\rule{.1in}{.1in}}
\providecommand{\U}[1]{\protect\rule{.1in}{.1in}}
\graphicspath{{pic/}}
\newtheorem{thm}{Theorem}

\newtheorem{lemma}{Lemma}

\newtheorem{remark}{Remark}

\begin{document}
	\title[Simulation of high-dimensional networks]{Efficient Steady-state
		Simulation of High-dimensional Stochastic Networks}
	\author[Blanchet, J., Chen, X., Glynn, P. and Si, N.]{Jose Blanchet, Xinyun Chen, Peter Glynn and Nian Si}
	\maketitle
	
	\begin{abstract}
		We propose and study an asymptotically optimal Monte Carlo estimator
		for steady-state expectations of a $d$-dimensional reflected Brownian
		motion. Our estimator is asymptotically optimal in the sense that it
		requires $\widetilde{O}(d)$ (up to logarithmic factors in $d$) i.i.d. Gaussian
		random variables in order to output an estimate with a controlled error. Our
		construction is based on the analysis of a suitable multi-level Monte Carlo
		strategy which, we believe, can be applied widely. This is the first
		algorithm with linear complexity (under suitable regularity conditions) for
		steady-state estimation of RBM as the dimension increases.
	\end{abstract}
	
	\renewcommand\footnoterule{}
	
	\section{Introduction}
	
	The complexity of supply chains and communication networks have resulted in
	ever increasing stochastic networks. On the other hand, the steady-state
	analysis of these systems is of significant interest because operators
	often focus on performance analysis or control of long term average rewards/costs per unit of time. These reasons motivate our focus in this paper,
	namely, the study of efficient Monte Carlo methods for steady-state analysis
	of high-dimensional stochastic networks.
	
	We consider a family of multidimensional reflected Brownian motion (RBM)
	living in the positive orthant. Under natural uniformity conditions, as the
	dimension increases, we propose a steady-state simulation estimator which is
	optimal in the sense of requiring almost a linear number of i.i.d. Gaussian random variables to output an estimate
	that is close to the steady-state expectation of the underlying RBM. We
	will provide an explicit description of the assumptions that we impose in
	Section \ref{sec: model}. It suffices to say at the moment that these
	conditions correspond basically to uniform stability and uniformly bounded
	variances.
	
	As far as we understand, this paper provides the first class of optimal steady-state estimators for a reasonably general class of
	stochastic networks as the dimension increases.
	
	RBM can be used to approximate the workload process of a wide range of
	stochastic networks in heavy traffic. In addition, RBM can be succinctly
	parameterized in terms of means, variances, and the routing
	architecture of the network. These properties make it an ideal vehicle for
	the study of Monte Carlo estimators for stochastic networks indexed by a set
	of parameters growing in the number of dimensions. Precisely, RBM is a parsimonious,
	yet powerful, stylized model capturing the features that make steady-state
	analysis of stochastic networks challenging. In particular, direct
	computation of the steady-state distribution for RBM is a very challenging
	problem, even in low dimensions. There is no closed-form expression in
	general and even numerical methods are difficult to apply. These difficulties
	arise from the fact that RBM is defined in terms of a system of constrained
	stochastic differential equations known as the Skorokhod problem which
	involves delicate local-time-like dynamics.
	
	Our analysis builds on recent work by  \cite
	{Budhiraja_2019} and \cite{BlanchetChen2017}. The first results showing a polynomial rate of convergence
	to steady-state for a high-dimensional RBM are given in \cite{BlanchetChen2017}%
	. The proof technique used in \cite{BlanchetChen2017} involves the following
	three ingredients: a) the use of a coupling between a steady-state version
	of the RBM and one starting from a given initial condition driven by the same
	Brownian motion; b) the application of results from \cite%
	{KellaRamasubramanian_2012} which leads to a contraction
	factor as the product of certain random matrices when hitting the constrain
	boundaries at a certain epochs; c) a Lyapunov bound which estimates the
	return times of the contraction epochs - basically the return time to the constrain
	boundaries. By combining a)-c), \cite{BlanchetChen2017} provided an estimate
	of the form $O\left( d^{4}\log ^{2}\left( d\right) \right) $ for the
	relaxation time (measured in terms of the Wasserstein distance) between an
	RBM starting from the origin and its steady-state distribution. The work of
	\cite{Budhiraja_2019} introduced a weighted Lyapunov function (i.e.
	modifying step c)), greatly improving these estimates and obtaining a
	relaxation time of $O\left( \log ^{2}\left( d\right) \right) $. This
	suggests simulating the RBM of interest for a time of $O\left( \log
	^{2}\left( d\right) \right) $ to control the size of the initial transient
	bias.
	
	In addition to dealing with the initial transient bias, numerical simulation also involves discretization bias. In particular, discretizing a one dimensional Brownian
	motion with a grid of size $\varepsilon $ induces an error of  $%
	\widetilde{O}\left( \varepsilon ^{1/2}\right) $ in uniform norm on compact
	intervals. (The tilde notation here means that we are
	ignoring logarithmic factors in $\log \left( 1/\varepsilon \right)$. Because RBM is a Lipschitz function of Brownian motion (the Lipschitz constant depends on $d$), we could combine all the above bias analysis for a $d$-dimensional RBM using the triangle inequality,
	resulting in an error bound of  $\widetilde{O}\left( d\varepsilon
	^{1/2}\right) $ on any given compact interval. Consequently, a
	back-of-the-envelop calculation suggests that direct simulation, even using
	the sharp analysis in  \cite{Budhiraja_2019}, yields a complexity of $%
	\varepsilon =\widetilde{O}\left( d^{-2}\right) $ in order to guarantee a
	controlled error. In turn, this yields that the
	overall number of Gaussian random variables simulated to obtain an estimate
	with controlled error is of $\widetilde{O}\left( d^{3}\right)
	$, which is superlinear in $d$. Our analysis in this paper, in contrast,
	shows that the estimation can actually be done in complexity $\widetilde{O}%
	\left( d\right) $ measured by the number of i.i.d. standard Gaussian random
	variables simulated.
	
	Moreover, because the results in \cite{Budhiraja_2019} build from the
	elements a)-c), the contraction estimate may be difficult to translate to
	other situations of interest, for example, in the analysis of other types of
	stochastic networks or general high-dimensional processes whose
	steady-state distribution may be of interest (e.g. Markov chain Monte Carlo
	models).
	
	This takes us to the main contributions of this paper, which are summarized
	as follows.
	
	I) First, we theoretically show that our simulation estimator approximates the steady-state
	distribution of the underlying RBM in $\widetilde{O}\left( d\right) $ time
	(i.e. almost linear time), measured in terms of i.i.d. Gaussian random
	variables generated.
	
	II) Second, we provide an alternative method to deriving the contraction
	estimates for the initial transient bias, which is based on the derivative of the underlying RBM with respect to the
	initial condition. The intuition is that the rate of convergence to stationarity is
	dictated by how fast the process `forgets' its initial condition, i.e. how fast the derivative with respect to the initial condition
	converges to zero. This approach, although analyzed explicitly only in the
	context of RBM in this paper, is, we believe, applicable to many other
	settings. 
	
	A key idea behind our first contribution is that for numerical simulation of a $d$-dimension RBM on finite time intervals, we analyze the
	contribution of discretization bias of the $d$ Brownian motions altogether instead of separately, in order to obtain a finer bound on the simulation bias.
	
	A crucial aspect in the development of II) is the use of derivative
	estimates of RBM with respect to the initial condition, using tools
	developed in \cite{MandelbaumRamanan_2010}. These derivatives, as it turns
	out, can be computed as the product of random matrices precisely arising in
	item b) mentioned earlier in the analysis of \cite{BlanchetChen2017}. This
	is both reassuring and convenient because we can simply take advantage of
	the analysis both in \cite{BlanchetChen2017} and \cite%
	{Budhiraja_2019}. However, studying the derivative process with respect to
	the initial condition is a type of strategy that can be applied in a wide
	range of settings of interest. So, we believe that the strategy deployed in
	this paper can be used as a blueprint for the development of efficient Monte
	Carlo methods for high-dimensional steady-state analysis in many other
	settings. These developments will be studied in future research.
	
	Our estimators are built using the multilevel Monte Carlo (MLMC) method (see
	\cite{Giles_2008}) in conjunction with the key idea discussed earlier and
	also the contraction property mentioned in II). For a review of multilevel
	Monte Carlo the reader is referred to \cite{Giles_15}. The MLMC method and
	its randomized variant, which can be used to remove bias under certain
	conditions (see \cite{Glynn_Rhee_15}), have been investigated both in the
	discretization of stochastic differential equations and, more recently, also
	in the context of steady-state expectations, see \cite{GMSVZ_19} and also
	\cite{Glynn_Rhee_14} )
	
	As in \cite{GMSVZ_19}, we are concerned both with the error in the numerical
	discretization of the underlying SDE and the time horizon contraction
	property. We both use a synchronous coupling, which, in our case, is
	motivated by the analysis in \cite{BlanchetChen2017}. A key difference,
	however, is that our goal is to study the complexity of the method as the
	dimension $d$ increases to infinity and showing that our estimator has
	essentially linear complexity in the dimension, as measured by the total
	number of generated random seeds. Indeed, we believe that this is also a key
	difference between our work and virtually every work to the date which uses
	multilevel Monte Carlo methods or steady-state Monte Carlo estimation in
	generic stochastic networks.
	
	
	The rest of the paper is organized as follows. In Section \ref{sec: model} we review
	the definition of RBM and discuss the uniformity conditions which we use to
	test the asymptotic optimality of our algorithm. The simulation algorithm is
	given in Section \ref{sec: algorithm}, together with the main result of this
	paper, Theorem \ref{thm: 1}. A numerical experiment that validates the
	theoretical performance of the algorithm, tested in the setting of networks
	of increasing size, is given in Section \ref{sec: numerics}. Finally, the proof of
	Theorem \ref{thm: 1} is given in Section \ref{sec:proof}.
	
	\section{Model and Assumptions}
	
	\label{sec: model}
	
	\subsection{Skorokhod Problem and RBM}
	
	A multidimensional reflected Brownian motion (RBM) can be defined as the
	solution to a Skorokhod problem with Brownian input. In particular, let $%
	\mathbf{X}(\cdot)$ be a multi-dimensional Brownian motion with drift vector $%
	\boldsymbol{\mu }$, covariance matrix $\Sigma :=CC^{T}$, and initial value $%
	\mathbf{X}(0)= 0$. Let $Q$ be a substochastic matrix, i.e. $Q\geq 0$ and all
	its row sums $\leq 1$, and define $R=(I-Q)^{T}$. We assume $R$ is an $M$%
	-matrix, i.e.
	\begin{equation}
	R^{-1}\text{ exists and it has non-negative entries.}  \label{M_Cond}
	\end{equation}%
	The seminal paper \cite{HarrisonReiman_1981} shows that the following
	Skorokhod problem \eqref{eq: Skorokhod} is well posed (i.e. it has a unique
	solution) in the case where the input $\mathbf{X}\left( \cdot \right) $ is
	continuous and $R$ is an $M$-matrix. \newline
	
	\textbf{Skorokhod Problem: }\textit{Given a process }$\mathbf{X}\left( \cdot
	\right) $\textit{\ and a matrix }$R$\textit{, we say that the pair }$(
	\mathbf{Y},\mathbf{L})$\textit{\ solves the associated Skorokhod problem if}
	\begin{equation}  \label{eq: Skorokhod}
	0\leq \mathbf{Y}\left( t\right) =\mathbf{Y}\left( 0\right) +\mathbf{X}\left(
	t\right) +R\mathbf{L}\left( t\right) ,~\mathbf{L}(0)=0
	\end{equation}
	\textit{where the }$i$\textit{-th entry of }$\mathbf{L}\left( \cdot \right) $
	\textit{\ is non-decreasing and }$\int_{0}^{t}Y_{i}\left( s\right)
	dL_{i}\left( s\right) =0$\textit{.}\newline
	
	When the input process $\mathbf{X}$ is a multi-dimensional Brownian motion
	with parameter $(\boldsymbol{\mu},\Sigma)$, we call the process $\mathbf{Y}%
	(\cdot) $ solved from \eqref{eq: Skorokhod} a $(\boldsymbol{\mu}, \Sigma,R)$%
	-RBM.
	
	\begin{remark}
		From the perspective that RBM $\mathbf{Y}(\cdot)$ is an approximation to the
		workload process of a stochastic network, the assumption that $R$ is an $M$%
		-matrix is equivalent to $Q^n\to 0$, i.e. the network is open in the sense
		that all jobs will eventually leave the network.
	\end{remark}
	
	For general Skorokhod problems, under the $M$-condition and some mild
	conditions on $\mathbf{X}\left( \cdot \right) $, the assumption that
	\begin{equation}
	R^{-1}E\mathbf{X}\left( 1\right) =R^{-1}\boldsymbol{\mu }<0,
	\label{Stability_Cond}
	\end{equation}
	implies that $\mathbf{Y}\left( t\right) \Rightarrow \mathbf{Y}\left( \infty
	\right) $ as $t\rightarrow \infty $, where $\mathbf{Y}\left( \infty \right) $
	is a random variable with the (unique) stationary distribution of $\mathbf{Y}%
	\left( \cdot \right) $. (We use $\Rightarrow $ to denote weak convergence.)
	In particular, according to \cite{HarrisonWilliams_1987b}, condition (\ref%
	{Stability_Cond}) is necessary and sufficient for stability of the $\left(
	\boldsymbol{\mu },\Sigma ,R\right) $-RBM (i.e. a unique stationary
	distribution exists) under the $M$-condition (\ref{M_Cond}) (see also \cite%
	{KellaRamasubramanian_2012} which studies necessary and sufficient
	conditions for more general types of input processes).
	
	\subsection{Assumptions}
	
	The goal of our simulation algorithm is to estimate the steady-state
	expectation of certain function value of a multi-dimension RBM. In
	particular, let $\left( \boldsymbol{\mu },\Sigma ,R\right) $ be the
	parameters of the RBM and $f\left( \cdot \right) $ be the function to be
	evaluated. To study the complexity of the algorithm as the number of
	dimension grows, we shall consider a family of $\left( \boldsymbol{\mu }%
	,\Sigma ,R\right) $-RBMs under certain uniformity assumptions for arbitrary
	dimension $d$, as in \cite{BlanchetChen2017}. Implicitly, $R$, $\boldsymbol{%
		\mu }$, and $\Sigma $ are indexed by their dimension. Now we state the
	uniformity conditions imposed throughout the paper.
	
	\textbf{A1) Uniform contraction:} We let $R=I-Q^{T}$, where $Q$ is
	substochastic and assume that there exists $\beta_{0}\in\left( 0,1\right) $
	and $\kappa_{0}\in\left( 0,\infty\right) $ independent of $d$ such that
	\begin{equation}
	\left\Vert \mathbf{1}^T Q^n\right\Vert _{\infty}\leq \kappa_0(1-\beta
	_{0})^n.  \label{Ass_A1}
	\end{equation}
	
	Under (\ref{Ass_A1}) we observe that
	\begin{equation*}
	\left\Vert R^{-1}\mathbf{1}\right\Vert _{\infty}\leq
	b_{1}:=\kappa_0/\beta_{0}<\infty.
	\end{equation*}
	
	\textbf{A2) Uniform stability:} We write $\mathbf{X}\left( t\right) =
	\boldsymbol{\mu}t+C\mathbf{B}\left( t\right) $, where $\mathbf{B}\left(
	t\right) =$ $(B_{1}\left( t\right) ,...,B_{d}\left( t\right) )^{T}$ and the $%
	B_{i}\left( \cdot\right) $'s are standard Brownian motions, and the matrix $%
	C $ satisfies $\Sigma=CC^{T}$. We assume that there exists $\delta _{0}>0$
	independent of $d$ such that
	\begin{equation*}
	R^{-1}\boldsymbol{\mu}<-\delta_{0}\mathbf{1}.
	\end{equation*}
	
	\textbf{A3) Uniform marginal variability:} Define $\sigma_{i}^{2}=\Sigma
	_{i,i}$ (i.e. the variance of the $i$-th coordinate of $\mathbf{X}$). We
	assume that there exists $b_{0}\in\left( 0,\infty\right) $, independent of $%
	d\geq1$, such that
	\begin{equation*}
	b_{0}^{-1}\leq\sigma_{i}^{2}\leq b_{0}.
	\end{equation*}
	
	\textbf{A4)} \textbf{Lipschitz functions: }Throughout the rest of the paper,
	we assume that the function $f:\mathbb{R}_{+}^{d}\rightarrow \mathbb{R}$ for
	which we shall estimate $E\left[ f(\mathbf{Y}(\infty ))\right] $ is
	Lipschitz continuous in $l_{\infty }$ norm, i.e. there exists a constant $%
	\mathcal{\ L}>0$ such that
	\begin{equation*}
	|f(\mathbf{y})-f(\mathbf{y}^{\prime })|\leq \mathcal{L}\Vert \mathbf{y}-
	\mathbf{y}^{\prime }\Vert _{\infty },\text{ for all }\mathbf{y},\mathbf{y}
	^{\prime }\in \mathbb{R}_{+}^{d}.
	\end{equation*}
	
	\begin{remark}
		A detailed discussion on the Assumptions A1) to A3) is given in Section 2.2
		of \cite{BlanchetChen2017}. Assumption A4) holds if $f$ is chosen to
		quantify the performance of a finite number of servers in the network, or
		when the performance measure of the system is scaled by $d$, for instance,
		the average workload in the servers.
	\end{remark}
	
	\section{Two-Parameter Multilevel Monte Carlo Algorithm}
	
	\label{sec: algorithm} Any simulation estimator for stationary expectations
	of RBM is bound to contain two types of sources of bias. The first one is
	the discretization error, due to the fact that we can only simulate discrete
	approximation of continuous Brownian paths. The second source of bias is
	the initial transient bias or non-stationary error, due to the fact that we can
	only simulate the RBM during a finite time horizon. We call our simulation
	method a two-parameter multilevel Monte Carlo (MLMC) algorithm because when
	constructing the MLMC estimator, we use two parameters $\gamma \in (0,1)$
	and $T>0$ to control the discretization and non-stationary errors, respectively.
	
	As in the classic MLMC algorithm (\cite{Giles_2008}), the precision of the
	MLMC estimator is controlled by the total number of levels $L$. Besides, we
	need to specify the initial state $\mathbf{y}_{0}$ to simulate the RBM
	paths. Given the parameter set $(\gamma ,T,L,\mathbf{y}_{0})$, plus the
	parameters $(\boldsymbol{\mu },\Sigma ,R)$ for the RBM, and the function $%
	f$ to evaluate, we now describe how to construct the two-parameter MLMC
	estimator for $E[f(\mathbf{Y}(\infty ))]$ and will summarize the whole
	procedure at the end of this section.
	
	Let $\mathbf{B}\left( t\right) =\left( B_{1}\left( t\right) ,...,B_{d}\left(
	t\right) \right) ^{T}\in\mathbb{R}^d$ be a standard Brownian with drift $%
	\mathbf{0}$ and covariance matrix $I$. Given parameter $\gamma \in \left(
	0,1\right)$, we denote $\mathbb{D}_{m}=\{0,\gamma ^{m},2\gamma ^{m},...\}$
	for any integer $m\geq 0$. For every $t\geq 0$, we define $t_{m}^{+}=\inf
	\{r\in \mathbb{D}_{m}:r> t\}$ and $t_{m}^{-}=\sup \{r\in \mathbb{D}%
	_{m}:r\leq t\} $ . Note that, following the definition, $t_{m}^-=t$ for $%
	t\in \mathbb{D}_m$. Define a discretization of the standard Brownian motion
	of level $m$ as $\mathbf{B}^m(t)=\left( B_{1}^{m}\left( t\right)
	,...,B_{d}^{m}\left( t\right) \right) ^{T}$ such that
	\begin{equation*}
	B_{i}^{m}\left( t\right) =B_{i}\left( t_{m}^{-}\right) +\left(
	t-t_{m}^{-}\right) \frac{B_{i}\left( t_{m}^{+}\right) -B_{i}\left(
		t_{m}^{-}\right) }{t_{m}^{+}-t_{m}^{-}}, \text{ for all }t\geq 0\text{ and }
	i =1, 2,...,d.
	\end{equation*}
	It is easy to see that $\mathbf{B}^m(\cdot)$ is continuous and piecewise
	linear, and $\mathbf{B}^{m}\left( t\right) =\mathbf{B}\left( t\right) $ for
	all $t\in \mathbb{D}_{m}$. The corresponding discretization of the Brownian
	motion $\mathbf{X}(\cdot)$ driving the RBM \eqref{eq: Skorokhod} is defined
	as
	\begin{equation*}
	\mathbf{X}^{m}\left( t\right) =\boldsymbol{\mu }t+C\mathbf{B}^{m}\left(
	t\right) .
	\end{equation*}
	For any $0\leq s\leq t<\infty$, we write $\mathbf{X}_{s:t}$ (resp. $\mathbf{%
		X }_{s:t}^{m}$ ) to denote the increment of $\mathbf{X}(\cdot)$ over $[s,t]$%
	, i.e. $\mathbf{X}_{s:t}=\{\mathbf{X}\left( s+u\right)-\mathbf{X}(s) :0\leq
	u\leq t-s\}$, (resp. $\mathbf{X}^m_{s:t}=\{\mathbf{X}^m\left( s+u\right)-
	\mathbf{X}^m(s) :0\leq u\leq t-s\}$). We use $\mathbf{Y}\left( t-s;\mathbf{y}
	,\mathbf{X}_{s:t}\right) $ (resp. $\mathbf{Y}^m\left( t-s;\mathbf{y},\mathbf{%
		X}_{s:t}\right) $) to denote the value of RBM driven by $\mathbf{X}_{s:t} $
	at time point $t-s$ given initial value $\mathbf{Y}\left( 0\right) =\mathbf{y%
	}$ (resp. $\mathbf{Y}^m\left( 0\right)= \mathbf{y} $). Following this
	notation, we have
	\begin{equation}  \label{eq: RBM paste}
	\begin{aligned} \mathbf{Y}\left( t+s;\mathbf{y},\mathbf{X}_{0:s+t}\right)
	&=\mathbf{Y}\left( t;\mathbf{Y}\left( s;\mathbf{y},\mathbf{X}_{0:s}\right)
	,\mathbf{X}_{s:s+t}\right),\\ \mathbf{Y}^m\left(
	t+s;\mathbf{y},\mathbf{X}_{0:s+t}\right) &=\mathbf{Y}^m\left(
	t;\mathbf{Y}^m\left( s;\mathbf{y},\mathbf{X}_{0:s}\right)
	,\mathbf{X}_{s:s+t}\right) . \end{aligned}
	\end{equation}
	
	To construct the multi-level estimator, we introduce an integer-valued
	random variable $M\in \{0,1,2,...,L-1\}$, where $L$ is the total number of
	levels. The random variable $M$ is independent of the process $\mathbf{X}%
	\left( \cdot \right) $ and follows probability distribution
	\begin{equation*}
	P(M=m)=p\left( m\right) =\gamma ^{m}\left( 1-\gamma \right) /(1-\gamma
	^{L})\triangleq K(\gamma )\gamma ^{m},\text{ for }0\leq m<L.
	\end{equation*}
	
	Now, we give the formal definition of the two-parameter MLMC estimator $Z$
	for $E[f(\mathbf{Y}(\infty))]$ with input parameter set $(\gamma, T, L,
	\mathbf{y}_0)$ as 
	
	\begin{equation}
	Z=\frac{1}{p(M)}\left( f\left( \mathbf{Y}^{M+1}\left( MT;\mathbf{Y}%
	^{M+1}\left( T; \mathbf{y}_{0},\mathbf{X}_{0:T}\right) ,\mathbf{X}%
	_{T:(M+1)T}\right) \right)-f\left( \mathbf{Y}^{M}\left( MT;\mathbf{y}_{0},%
	\mathbf{X}_{T:(M+1)T}\right) \right) \right) +f\left( \mathbf{y}_{0}\right) .
	\label{eq: MLMC}
	\end{equation}
	To see that $Z$ is indeed a good estimator for $E[f(\mathbf{Y}(\infty ))]$,
	we compute
	\begin{eqnarray*}
		&&E[Z]=E\left[ E\left[ Z|M\right] \right] \\
		&=&\sum_{m=0}^{L-1}\left( E\left[ f\left( \mathbf{Y}^{m+1}\left( mT;\mathbf{Y%
		}^{m+1}(T;\mathbf{y}_{0},\mathbf{X}_{0:T}),\mathbf{X}_{T:(m+1)T}\right)
		\right) \right] \right. \\
		&&\left. -E[f\left( \mathbf{Y}^{m}\left( mT;\mathbf{y}_{0},\mathbf{X}%
		_{T:(m+1)T}\right) \right) ]\right) +f\left( \mathbf{y}_{0}\right) \\
		&=&\sum_{m=0}^{L-1}\left( E[f\left( \mathbf{Y}^{m+1}\left( (m+1)T;\mathbf{y}%
		_{0},\mathbf{X}_{0:(m+1)T}\right) \right) ]-E[f\left( \mathbf{Y}^{m}\left(
		mT;\mathbf{y}_{0},\mathbf{X}_{0:mT}\right) \right) ]\right) +f\left( \mathbf{%
			y}_{0}\right) \\
		&=&E[f\left( \mathbf{Y}^{L}\left( TL;\mathbf{y}_{0},\mathbf{X}_{0:LT}\right)
		\right) ].
	\end{eqnarray*}%
	The last equality holds because $\mathbf{Y}^{m}\left( mT;\mathbf{y}_{0},%
	\mathbf{X}_{0:mT}\right) =\mathbf{y}_{0}$ for $m=0$. Consequently, we can
	split the estimation bias into two parts:
	\begin{equation}
	\begin{aligned}&E[f\left( \mathbf{Y}^{L}\left( TL ; \mathbf{y}_0,
	\mathbf{X}_{0:LT}\right) \right) ]-E[f(\mathbf{Y}(\infty))]\\
	=~&\left(E[f\left( \mathbf{Y}^{L}\left( TL ; \mathbf{y}_0,
	\mathbf{X}_{0:LT}\right) \right) ]-E[f\left( \mathbf{Y}\left( TL ;
	\mathbf{y}_0, \mathbf{X}_{0:LT}\right) \right) ]\right) \\
	~&+\left(E[f\left( \mathbf{Y}\left( TL ; \mathbf{y}_0,
	\mathbf{X}_{0:LT}\right) \right) ]-E[f(\mathbf{Y}(\infty))]\right)\\ =~&
	\text{Discretization Error} + \text{Non-stationarity Error}. \end{aligned}
	\label{eq: MSE}
	\end{equation}%
	Intuitively, as $L\rightarrow \infty $, the two errors will both go to 0,
	and as a consequence, we can obtain accurate estimate of $E[f(\mathbf{Y}%
	(\infty ))]$ by taking $L$ large enough. In Section \ref{sec: d error bound}
	and \ref{sec:non_stationary_bd}, we shall provide theoretical upper bounds
	for those two errors in terms of $L$, and also analyze their dependence on
	the number of dimension $d$. Then, we apply these theoretical error bounds
	to control the mean square error (MSE) of the simulation estimator, and
	obtain the main complexity analysis result for our simulation algorithm in
	Section \ref{sec: complexity}.
	
	The above description of the two-parameter multilevel
	Monte Carlo method is summarized in Algorithm \ref{alg: 1}. The main result of the paper as follows. We show that, under proper choice of algorithm hyperparameters, the computational budget for Algorithm \ref{alg: 1} to obtain estimator of a fixed accuracy level is almost
	linear in the dimension $d$. The proof relies on a sequence of analysis on the dimension dependence of the discretization and non-stationary error in the simulation procedures, and will be given in Section \ref{sec:proof}.
	
	\begin{thm}
		\label{thm: 1} Suppose $\mathbf{Y}$ (indexed by the number of dimensions $d$
		) is a sequence of RBM satisfying Assumption 1-4. Then, the total expected
		cost, in terms of \textbf{the number of random seeds}, for the 2-dimensional
		MLMC Algorithm \ref{alg: 1} to produce an estimator of $E[f(\mathbf{Y}
		(\infty ))]$ with mean square error (MSE) $\varepsilon ^{2}$ is
		\begin{equation*}
		O\left(\varepsilon^{-2}d\log(d)^3(\log(\log(d))+\log(1/\varepsilon))^3
		\right).
		\end{equation*}
	\end{thm}

	\noindent
	\renewcommand\footnoterule{}
	\begin{algorithm}[t]
		\caption{Two-Parameter Multilevel Monte Carlo for RBM}
		\label{alg: 1}
		\begin{flushleft}
			\hspace*{0.02in} {\bf Input:}\\
			The parameters of the RBM: $(\boldsymbol{\mu},\Sigma,R)$;\\
			The function to evaluate: $f:\mathbb{R}^d_+\to\mathbb{R}$;\\
			The target error level $\epsilon$;
			
			\hspace*{0.02in} {\bf Output:}\\
			An estimator for $E[\mathbf{Y}(\infty)]$, $\bar{Z};$ \\
			\hspace*{0.02in} {\bf Hyperparameter Setting:}\\
			Step size: $1>\gamma >0$;
			\footnote{~We recommend choosing step size $\gamma$ around 0.05, but our algorithm is not sensitive to the specific choice of $\gamma$.}
			\\
			Path length: $T=O(\log(d)^2)$;\\
			Number of levels: $L = \lceil \left( \log (\log (d))+2\log (1/\varepsilon
			)+k_{1}\right)/\log(1/\gamma) \rceil $, for a numerical constant $k_1$;\\
			Initial value: $\mathbf{y}_0 = \mathbf{0}$;\\
			Simulation rounds: $N = \lceil K(\gamma)^{-1}\gamma^{-L}L \rceil $, where $K(\gamma)= (1-\gamma) /(1-\gamma^{L})$;\\
			\hspace*{0.02in} {\bf Algorithm procedure:}\\
		\end{flushleft}	
		\begin{algorithmic}[1]
			\For{$i=1$ to $N$}
			\State Generate $M$ with $P(M=m) = p(m)=K(\gamma)\gamma^m$;
			\State Simulate a discrete Brownian path $\mathbf{B}^{M+1}(t)$ with step size $\gamma^{M+1}$ on $[0,(M+1)T]$;
			\State Compute $\mathbf{B}^{M}(t)$ as a discrete Brownian path such that $\mathbf{B}^{M}(t)=\mathbf{B}^{M+1}(t)$ for all $t\in\mathbb{D}_{M}$;
			\State Compute
			$$X^M(t) =\boldsymbol{\mu }t+C\mathbf{B}^M(t)\text{ and }X^{M+1}(t) =\boldsymbol{\mu }t+C\mathbf{B}^{M+1}(t);$$
			\State Compute
			$$Z_i=\frac{1}{p(M)}\left(f(\mathbf{Y}^{M+1}((M+1)T,\mathbf{y}_0,\mathbf{X}_{0:(M+1)T}))-f(\mathbf{Y}^{M}(MT,\mathbf{y}_0,\mathbf{X}_{T:(M+1)T}))\right);$$
			\EndFor
			\Return $\bar{Z} =f(\mathbf{y}_0)+\frac{1}{N}\sum_{i=1}^N Z_i$.
		\end{algorithmic}
	\end{algorithm}

	\section{Numerical Experiments}\label{sec: numerics}
	We test the theoretical performance guarantee (i.e. Theorem \ref{thm: 1}) of our algorithm using the so-called symmetric RBMs.
	In this case, the true value of $E[Y_{1}(\infty )]$ is known with
	closed-form expression so that we can check the dimension dependence of the
	simulation MSE and complexity. To do this, we consider a sequence of
	symmetric RBMs of different dimensions from 5 up to 200. In detail, for each
	$d\in\{5,6,...,200\}$, the covariance matrix takes the form
	\begin{equation*}
	\Sigma =\left[
	\begin{array}{cccc}
	1 & \rho _{\sigma } & \ldots & \rho _{\sigma } \\
	\rho _{\sigma } & 1 & \ldots & \rho _{\sigma } \\
	\vdots &  & 1 & \vdots \\
	\rho _{\sigma } & \ldots & \rho _{\sigma } & 1%
	\end{array}
	\right] ,
	\end{equation*}
	and the reflection matrix takes the form
	\begin{equation*}
	R=\left[
	\begin{array}{cccc}
	1 & -r & \ldots & -r \\
	-r & 1 & \ldots & -r \\
	\vdots &  & 1 & \vdots \\
	-r & \ldots & -r & 1%
	\end{array}
	\right] .
	\end{equation*}
	To be consistent with Assumptions A1) to A3), we pick
	\begin{equation*}
	\rho _{\sigma }=-\frac{1-\beta }{d-1}\text{ and }r=\frac{1-\beta }{d-1},
	\end{equation*}
	for given $0<\beta <1$. According to \cite{DaiHarrison_1992}, the
	steady-state expectation of workload in each station equals to
	\begin{equation*}
	E[Y_{1}(\infty )]=\frac{1-(d-2)r+(d-1)r\rho _{\sigma }}{2(1+r)}=\frac{\beta
	}{2}.
	\end{equation*}
	For $\beta =0.8$, the true value of $E[Y_{1}(\infty )]=0.4$.
	
	In the first group of numerical experiments, we compare the algorithm
	performance for different choices of parameter $\gamma\in\{0.01, 0.05, 0.1\}$
	at target error level $\varepsilon = 0.01$. The other parameters are as
	follows: $T=\log(d)^2/2,L=\lceil \left( \log (\log (d))+2\log (1/\varepsilon
	)-2\right)/\log(1/\gamma)\rceil$, and $N=\lceil K(\gamma)^{-1}\gamma^{-L}L
	\rceil $. Figure \ref{plot:rbm} shows the estimated mean and total
	complexity across dimensions from $d=5$ to $d=200$ for different choices of $%
	\gamma$. It shows that most of the absolute error fluctuates around 0.01 and
	the total complexity grows approximately linear in the number of dimension
	for all three values of $\gamma$. The simulation error is not sensitive to
	the choice of $\gamma$. Besides, the complexity is best when $\gamma = 0.05$%
	, as indicated by our theoretic analysis (Lemma \ref{lmm: gamma}).
	
	In our second group of numerical experiments, we aim to show that our choice
	of the parameters is optimal in the sense that the precision level of the
	algorithm is stable across different number of dimensions. In particular, we
	estimate the mean square error~(MSE) of the estimators for $\gamma = 0.05$
	and target error level $\varepsilon=0.05$ with the other parameters remain
	the same. For each dimension range from $\{10,20,30,\ldots,200\}$, we
	generate 250 estimators to estimate the MSE as well as the 95\% confidence
	band of the MSE and the results are reported in Figure \ref{fig:mse}. We see
	the MSE is stable around $5\times 10^{-4}$ across different dimensions,
	which is smaller than the target level $\varepsilon^2 = 0.0025$.
	
	\begin{figure}[]
		\centering
		\subfigure[$\gamma =0.01$]{
			\label{err} \includegraphics[width=12cm]{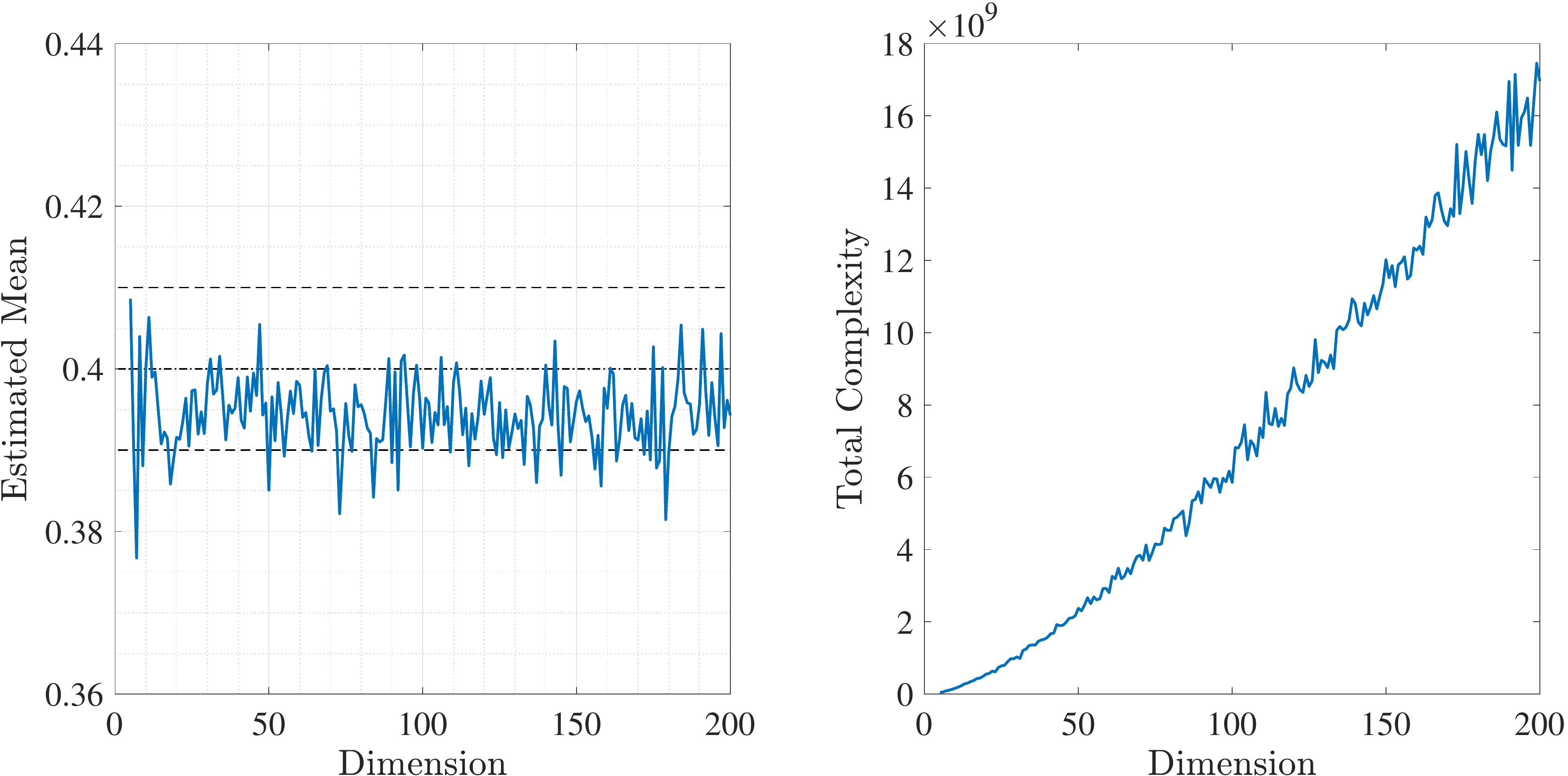}}
		\subfigure[$\gamma =0.05$]{
			\label{err} \includegraphics[width=12cm]{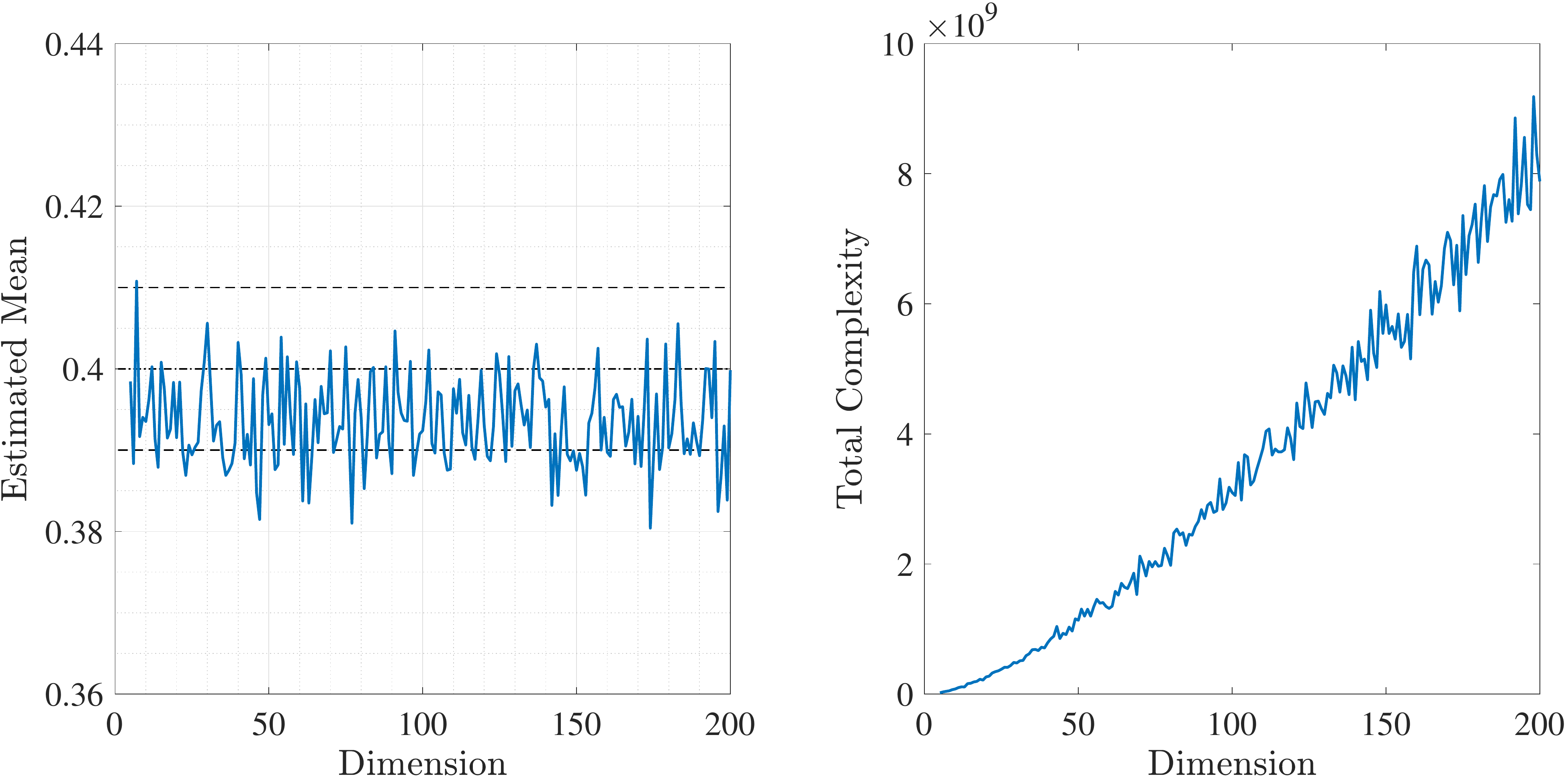}}
		\subfigure[$\gamma =0.1$]{
			\label{comp} \includegraphics[width=12cm]{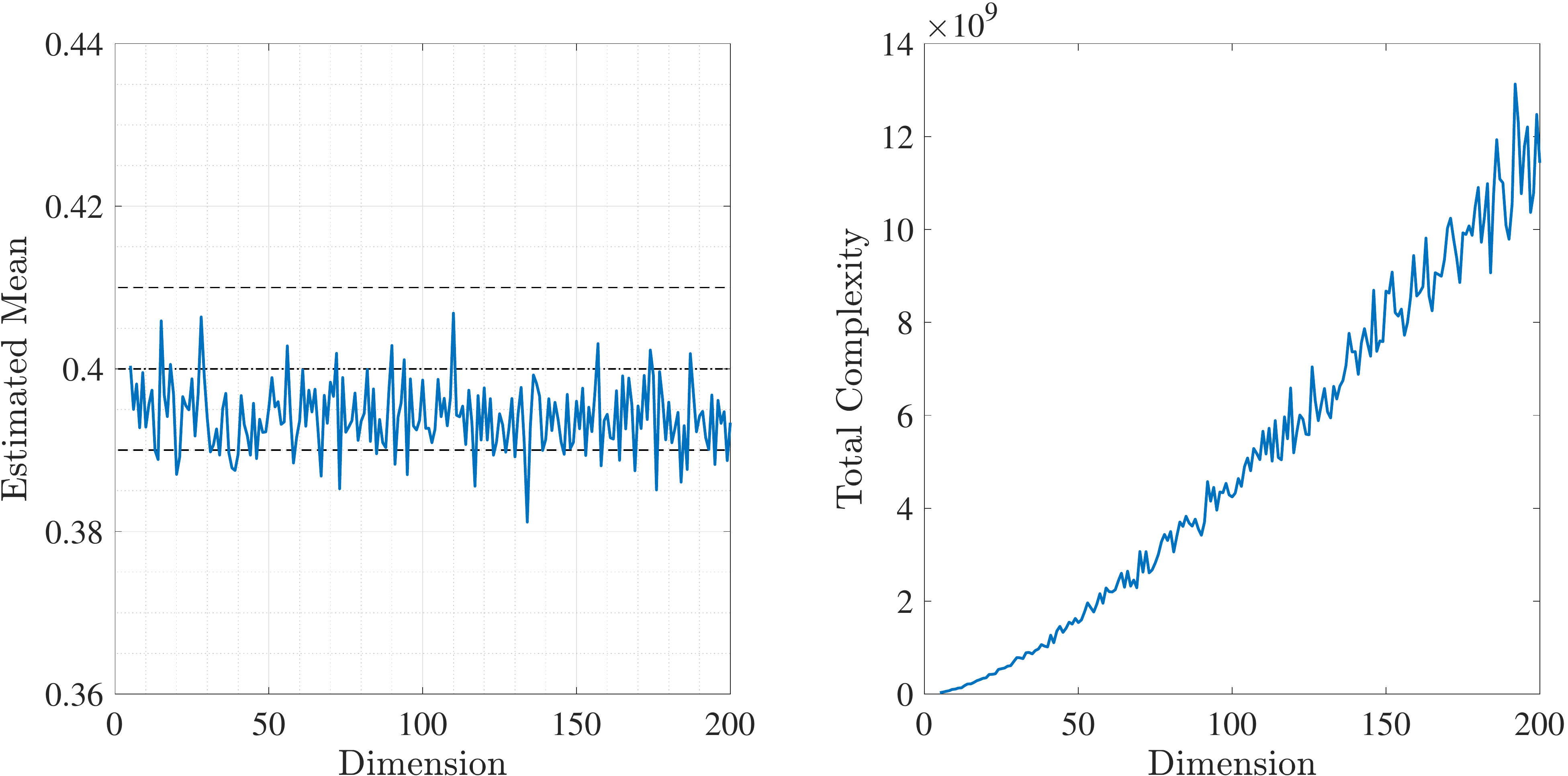}}
		\caption{Simulation results for symmetric RBMs.}
		\label{plot:rbm}
	\end{figure}
	
	\begin{figure}[ht]
		\centering
		\includegraphics[width=10cm]{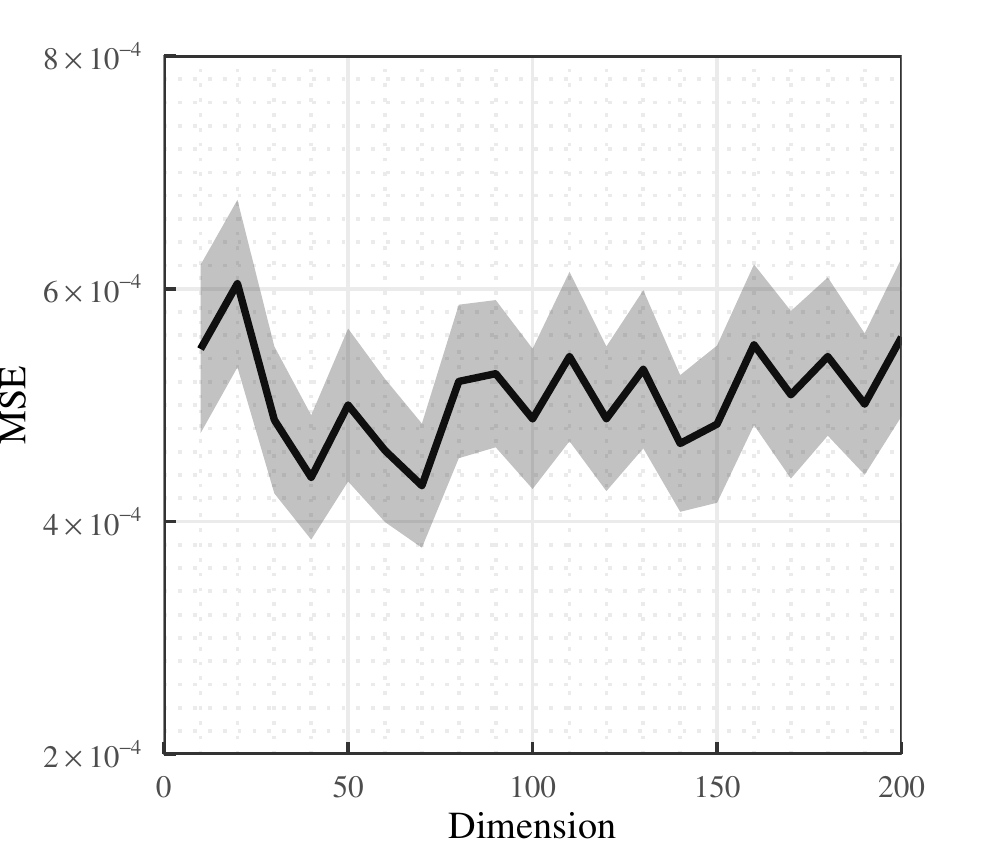}
		\caption{Mean square error of the estimators. The shade represents 95\% confidence band of the MSE.}
		\label{fig:mse}
	\end{figure}
	
	\section{Proof of Theorem \protect\ref{thm: 1}}
	
	\label{sec:proof} In this section, we develop theoretic MSE bounds of the
	2-parameter MLMC estimator in terms of the number of dimensions $d$ using
	the hyperparameters $(\gamma ,T,L, \mathbf{y}_{0})$ specified in the
	algorithm \ref{alg: 1}. As in \eqref{eq:
		MSE}, the estimation bias of the 2-parameter MLMC estimator $Z$ can be split
	into two parts corresponding to the discretization error and non-stationary
	error. The sketch of the proof is as follows:
	
	\begin{enumerate}
		\item In Section \ref{sec: d error bound}, we derive an upper bounds for the
		discretization error in Lemma \ref{lm: discrete error RBM}, which is based
		on the discretization error for Brownian motion (Lemma \ref{lm: discrete
			error BM}) and an explicit upper bound for the Lipschitz constant of
		Skorokhod mapping (Lemma \ref{lemma: lpsz}).
		
		\item In Section \ref{sec:non_stationary_bd}, we provide a bound for the
		non-stationary error in Lemma \ref{lemma: stationary error} by analyzing the
		derivative of RBM with respect to its initial value (Lemma \ref{lemma:
			derivative}).
		
		\item Finally, in Section \ref{sec: complexity}, we derive a theoretic upper
		bound for the algorithm complexity based using the error bounds.
	\end{enumerate}
	
	\subsection{Discretization Error Bounds}
	
	\label{sec: d error bound}
	
	To bound the discretization error, we first bound the discretization error
	of multi-dimensional Brownian motion.
	
	\begin{lemma}
		\label{lma:max_gaussian} Suppose $Z_{1},Z_{2},\ldots ,Z_{n}$ are Gaussian
		variables (not necessarily independent) with mean 0 and variance 1. Then, we
		have $E\left[ \max_{1\leq i\leq n}Z_{i}^{2}\right] \leq 4\left( \log
		n+1/2\log (2)\right) . $
	\end{lemma}
	
	\begin{proof}[Proof of Lemma \protect\ref{lma:max_gaussian}]
		For $\lambda \in (0,1/2),$ we have
		\begin{eqnarray*}
			E\left[ \max_{1\leq i\leq n}Z_{i}^{2}\right] &=&\frac{1}{\lambda }E\left[
			\log \left( \exp \left( \lambda \max_{1\leq i\leq n}Z_{i}^{2}\right) \right) %
			\right] \\
			&\leq &\frac{1}{\lambda }\log E\left[ \exp \left( \lambda \max_{1\leq i\leq
				n}Z_{i}^{2}\right) \right] \\
			&\leq &\frac{1}{\lambda }\log E\left[ \sum_{i=1}^{n}\exp \left( \lambda
			Z_{i}^{2}\right) \right] \\
			&=&\frac{1}{\lambda }\left( \log n-1/2\log (1-2\lambda \right) ).
		\end{eqnarray*}
		We can pick $\lambda =1/4$ and then
		\begin{equation*}
		E\left[ \max_{1\leq i\leq n}Z_{i}^{2}\right] \leq 4\left( \log n+1/2\log
		(2)\right) .
		\end{equation*}
	\end{proof}
	
	\begin{lemma}
		\label{lm: discrete error BM} For $0<\gamma <1$ and $m\geq 1$, let $\mathbf{X%
		}^{m}(\cdot)$ be a discretized $d$-dimension Brownian path with step size $%
		\gamma ^{m}$. Then, there exists a positive constant $C_0$, such that for any
		$d\geq 1$,$m\geq 1$, $t>0$,
		\begin{equation*}
		E[\max_{1\leq i\leq d}\max_{0\leq s\leq t}(X_{i}^{m}(s)-X_{i}(s))^{2}]\leq
		C_{0}\gamma ^{m}(\log (t)+\log (d)+m\log (1/\gamma )).
		\end{equation*}
	\end{lemma}
	
	\begin{proof}[Proof of Lemma \protect\ref{lm: discrete error BM}]
		Let $\mathbf{\tilde{X}}\left( t\right) =\mathbf{X}\left( t\right) -
		\boldsymbol{\mu }t$ and $\mathbf{\tilde{X}}^{m}\left( t\right) =\mathbf{X}
		^{m}\left( t\right) -\boldsymbol{\mu }t$. Note that
		\begin{eqnarray*}
			&&\max_{1\leq i\leq d}\max_{0\leq s\leq t}\left( X_{i}^{m}\left( s\right)
			-X_{i}(s)\right) ^{2} \\
			&\leq &\max_{1\leq i\leq d}\max_{0\leq s\leq \gamma ^{m}\left\lceil t/\gamma
				^{m}\right\rceil }\left( X_{i}^{m}\left( s\right) -X_{i}(s)\right) ^{2} \\
			&=&\max_{1\leq i\leq d}\max_{0\leq k\leq \left\lfloor t/\gamma
				^{m}\right\rfloor }\max_{0\leq s\leq \gamma ^{m}}\left( \tilde{X}_{i}\left(
			\gamma ^{m}k+s\right) -\tilde{X}_{i}^{m}\left( \gamma ^{m}k+s\right) \right)
			^{2}.
		\end{eqnarray*}
		For $0\leq s<\gamma ^{m}$ and $0\leq k\leq \left\lfloor t/\gamma
		^{m}\right\rfloor ,$ we have
		\begin{eqnarray*}
			&&\left( \tilde{X}_{i}\left( \gamma ^{m}k+s\right) -\tilde{X}_{i}^{m}\left(
			\gamma ^{m}k+s\right) \right) ^{2} \\
			&\leq &\max \left\{ \left( \tilde{X}_{i}\left( \gamma ^{m}k+s\right) -\tilde{
				X}_{i}\left( \gamma ^{m}k\right) \right) ^{2},\left( \tilde{X}_{i}\left(
			\gamma ^{m}k+\gamma ^{m}\right) -\tilde{X}_{i}\left( \gamma ^{m}k+s\right)
			\right) ^{2}\right\} \\
			&\leq &\left( \tilde{X}_{i}\left( \gamma ^{m}k+s\right) -\tilde{X}_{i}\left(
			\gamma ^{m}k\right) \right) ^{2}+\left( \tilde{X}_{i}\left( \gamma
			^{m}k+\gamma ^{m}\right) -\tilde{X}_{i}\left( \gamma ^{m}k+s\right) \right)
			^{2}.
		\end{eqnarray*}
		By time-reversibility of the Brownian process, we have
		\begin{equation*}
		\left( \tilde{X}_{i}\left( \gamma ^{m}k+s\right) -\tilde{X}_{i}\left( \gamma
		^{m}k\right) \right) ^{2}\overset{d}{=}\left( \tilde{X}_{i}\left( \gamma
		^{m}k+\gamma ^{m}\right) -\tilde{X}_{i}\left( \gamma ^{m}k+(\gamma
		^{m}-s)\right) \right) ^{2}.
		\end{equation*}
		Then, according to the increment independence of Brownian motion, we have
		\begin{eqnarray*}
			&&\max_{1\leq i\leq d}\max_{0\leq k\leq \left\lfloor t/\gamma
				^{m}\right\rfloor }\max_{0\leq s\leq \gamma ^{m}}\left( \tilde{X}_{i}\left(
			\gamma ^{m}k+s\right) -\tilde{X}_{i}\left( \gamma ^{m}k\right) \right) ^{2}
			\\
			&\overset{d}{=}&\gamma ^{m}\max_{1\leq i\leq d}\max_{0\leq k\leq
				\left\lfloor t/\gamma ^{m}\right\rfloor }\max_{0\leq s\leq 1}\left( \tilde{X}
			_{i}^{(k)}\left( s\right) \right) ^{2},
		\end{eqnarray*}
		where $\mathbf{\tilde{X}}^{(0)},\mathbf{\tilde{X}}^{(1)}\ldots $ are i.i.d.
		copies of $\mathbf{\tilde{X}}$ and $\mathbf{\tilde{X}}^{(k)}\mathbf{=}
		\left\{ \tilde{X}_{1}^{(k)},\tilde{X}_{2}^{(k)},\ldots \tilde{X}
		_{d}^{(k)}\right\} .$ Recall that (e.g. \cite{karlin1981second}, page 346)
		\begin{equation*}
		\max_{0\leq s\leq 1}\left( \tilde{X}_{i}^{(k)}\left( s\right) \right) ^{2}
		\overset{d}{=}\left( \tilde{X}_{i}^{(k)}\left( 1\right) \right) ^{2}.
		\end{equation*}
		Then, by Lemma \ref{lma:max_gaussian}, we have for $d>1$,
		\begin{eqnarray*}
			&&E\left[ \max_{1\leq i\leq d}\max_{0\leq s\leq \gamma ^{m}\left\lceil
				t/\gamma ^{m}\right\rceil }\left( X_{i}^{m}\left( s\right) -X_{i}(s)\right)
			^{2}\right] \\
			&\leq &2\gamma ^{m}E\left[ \max_{1\leq i\leq d}\max_{0\leq k\leq
				\left\lfloor t/\gamma ^{m}\right\rfloor }\left( \tilde{X}_{i}^{(k)}\left(
			s\right) \right) ^{2}\right] \\
			&\leq &2b_{0}\gamma ^{m}\left( 4\log \left( d\left\lceil t/\gamma
			^{m}\right\rceil \right) +2\log (2)\right) \\
			&\leq &C_{0}\gamma ^{m}(\log (t)+\log (d)+m\log (1/\gamma )).
		\end{eqnarray*}
	\end{proof}
	
	The following Lemma \ref{lemma: lpsz} shows that the Skorokhod mapping, from
	$\mathbf{X}$ to $\mathbf{Y}$, is Lipschitz continuous and provides an
	uniform upper bound for the Lipschitz constant. As a result, the
	discretization error $\sup_{0\leq s\leq T}\left\Vert \mathbf{Y}(s)- \mathbf{Y%
	}^{m}(s)\right\Vert _{\infty }$ can be uniformly bounded.
	
	\begin{lemma}
		\label{lemma: lpsz} Suppose $\mathbf{Y}(t)$ and $\mathbf{Y}^{\prime }(t)\in
		\mathbb{R}^{d}$ are the solutions to two Skorokhod problems
		\eqref{eq:
			Skorokhod} with the same reflection matrix $R$ satisfying Assumption A1),
		and input processes $\mathbf{X}(t)$ and $\mathbf{X}^{\prime }(t)$
		respectively for $t\in \lbrack 0,T]$. Then,
		\begin{equation*}
		\left\vert \mathbf{Y}(T)-\mathbf{Y}^{\prime }(T)\right\vert \leq
		2R\sup_{0\leq s\leq T}|\mathbf{X}(s)-\mathbf{X}^{\prime }(s)|.
		\end{equation*}
		As a direct consequence, under Assumptions A1) to A3), we have
		\begin{equation*}
		\|\mathbf{Y}(T)-\mathbf{Y}^{\prime }(T)\|_\infty \leq \frac{2\kappa_0}{\beta}
		\sup_{0\leq s\leq T}\|\mathbf{X}(s)-\mathbf{X}^{\prime }(s)\|_\infty.
		\end{equation*}
	\end{lemma}
	
	\begin{proof}[Proof of Lemma \protect\ref{lemma: lpsz}]
		The proof uses the fixed-point representation of the Skorokhod mapping as
		constructed in the proof of Theorem 1 in \cite{HarrisonReiman_1981}. In
		detail, we first need to do a transform on the space $\mathbb{R}^{d}$ with
		respect to a diagonal matrix $\Theta $ with positive diagonal elements,
		which depends only on $R$, such that $(\Theta \mathbf{Y},\Theta \mathbf{L})$
		($(\Theta \mathbf{Y}^{\prime },\Theta \mathbf{L}^{\prime })$) is the
		solution to a new Skorokhod problem of the form \eqref{eq: Skorokhod} with
		input process $\Theta \mathbf{X}$ ($\Theta \mathbf{X}^{\prime }$) and
		reflection matrix $R^{\ast }=I-(\Theta ^{-1}Q\Theta )^{T}$.(Note that in our
		notation, all vectors are column vectors while in \cite{HarrisonReiman_1981}%
		, they are treated as row vectors.)
		
		Let $Q^{\ast }=\Theta ^{-1}Q\Theta $. Then, according to \cite%
		{HarrisonReiman_1981}, the amount of reflection $\Theta \mathbf{L}$ and $%
		\Theta \mathbf{L}^{\prime }$ solve the following fixed point problem:
		\begin{equation*}
		\Theta \mathbf{L}(t)=\sup_{0\leq s\leq t}\left( Q^{\ast T}\Theta \mathbf{L}
		(s)-\Theta \mathbf{X}\right) ^{+}\text{ and }\Theta \mathbf{L}^{\prime
		}(t)=\sup_{0\leq s\leq t}\left( Q^{\ast T}\Theta \mathbf{L}^{\prime
		}(s)-\Theta \mathbf{X}^{\prime }\right) ^{+}\text{ for all }0\leq t\leq T.
		\end{equation*}
		Here the supreme is taken coordinate by coordinate. Since the elements of $%
		Q^{\ast T}$ are non-negative, we have
		\begin{equation*}
		\Theta (\mathbf{L}(t)-\mathbf{L}^{\prime }(t))\leq Q^{\ast T}\Theta
		\sup_{0\leq s\leq t}|\mathbf{L}(s)-\mathbf{L}^{\prime }(s)|+\sup_{0\leq
			s\leq t}\Theta |\mathbf{X}(s)-\mathbf{X}^{\prime }(s)|.
		\end{equation*}
		The inequality here also holds coordinate by coordinate. As $\Theta $ is a
		diagonal matrix with positive diagonal elements, we have
		\begin{align*}
		\left( \mathbf{L}(t)-\mathbf{L}^{\prime }(t)\right) & \leq \Theta
		^{-1}Q^{\ast T}\Theta \sup_{0\leq s\leq t}|\mathbf{L}(s)-\mathbf{L}^{\prime
		}(s)|+\sup_{0\leq s\leq t}|\mathbf{X}(s)-\mathbf{X}^{\prime }(s)| \\
		& =Q^{T}\sup_{0\leq s\leq t}|\mathbf{L}(s)-\mathbf{L}^{\prime
		}(s)|+\sup_{0\leq s\leq t}|\mathbf{X}(s)-\mathbf{X}^{\prime }(s)|.
		\end{align*}
		As a result,
		\begin{equation*}
		\sup_{0\leq s\leq T}|\mathbf{L}(s)-\mathbf{L}^{\prime }(s)|\leq
		Q^{T}\sup_{0\leq s\leq T}|\mathbf{L}(s)-\mathbf{L}^{\prime }(s)|+\sup_{0\leq
			s\leq T}|\mathbf{X}(s)-\mathbf{X}^{\prime }(s)|.
		\end{equation*}
		Since $(I-Q^{T})^{-1}=R^{-1}$ has non-negative elements, we have
		\begin{equation*}
		\sup_{0\leq s\leq T}|\mathbf{L}(s)-\mathbf{L}^{\prime }(s)|\leq
		R^{-1}\sup_{0\leq s\leq T}|\mathbf{X}(s)-\mathbf{X}^{\prime }(s)|.
		\end{equation*}
		In the end, we have
		\begin{align*}
		\sup_{0\leq s\leq T}|\mathbf{Y}(s)-\mathbf{Y}^{\prime }(s)|& \leq
		\sup_{0\leq s\leq T}|\mathbf{X}(s)-\mathbf{X}^{\prime }(s)|+|R|\sup_{0\leq
			s\leq T}|\mathbf{L}(s)-\mathbf{L}^{\prime }(s)| \\
		& \leq \sup_{0\leq s\leq T}|\mathbf{X}(s)-\mathbf{X}^{\prime
		}(s)|+|R|R^{-1}\sup_{0\leq s\leq T}|\mathbf{X}(s)-\mathbf{X}^{\prime }(s)|.
		\end{align*}
		Let us denote $R^{-1}$ by $S$, then $S_{ij}\geq 0$ for all $1\leq i,j\leq d$
		. Based on the fact that $R_{ii}=1$, $R_{ij}\leq 0$ for all $1\leq i\neq
		j\leq d$ and $\sum_{k}R_{ik}S_{ki}=1$ for all $1\leq i\leq d$, we have
		\begin{equation*}
		(|R|S)_{ii}=\sum_{k=1}^{d}|R_{ik}|S_{ki}=R_{ii}S_{ii}-\sum_{k\neq
			i}R_{ik}S_{ki}=R_{ii}S_{ii}+(-1+R_{ii}S_{ii})=2S_{ii}-1.
		\end{equation*}
		Note that $2S_{ii}-1>0$ as all diagonal elements of $R^{-1}\geq 1$.
		Similarly, as $\sum_{k}R_{ik}S_{kj}=0$ for all $1\leq i\neq j\leq d$, we
		have
		\begin{equation*}
		(|R|S)_{ij}=\sum_{k=1}^{d}|R_{ik}|S_{kj}=R_{ii}S_{ij}-\sum_{k\neq
			i}R_{ik}S_{kj}=R_{ii}S_{ij}+R_{ii}S_{ij}=2S_{ij}.
		\end{equation*}
		Therefore, $|R|R^{-1}=2R^{-1}-I$ where $I$ is the identity matrix of
		dimension $d$, and we conclude
		\begin{align*}
		\sup_{0\leq s\leq T}|\mathbf{Y}(s)-\mathbf{Y}^{\prime }(s)|& \leq
		\sup_{0\leq s\leq T}|\mathbf{X}(s)-\mathbf{X}^{\prime
		}(s)|+|R|R^{-1}\sup_{0\leq s\leq T}|\mathbf{X}(s)-\mathbf{X}^{\prime }(s)| \\
		& =2R^{-1}\sup_{0\leq s\leq T}|\mathbf{X}(s)-\mathbf{X}^{\prime }(s)|.
		\end{align*}
		Recall Assumption A1, $\left\Vert R^{-1}\mathbf{1}\right\Vert _{\infty }\leq
		\kappa _{0}/\beta _{0},$ the desired result follows.
	\end{proof}
	
	Given Lemma \ref{lm: discrete error BM} and Lemma \ref{lemma: lpsz}, we now
	are ready to provide a theoretic upper bound for the discretization error.
	
	
	\begin{lemma}
		\label{lm: discrete error RBM} For fixed $\gamma\in (0,1)$, $t>0$ and the
		number of dimensions $d$, there exists a positive constant $C_1$ such that
		\begin{equation*}
		E[\|\mathbf{Y}^m(t)-\mathbf{Y}(t)\|_\infty^2]\leq
		C_1\gamma^m\left(\log(t)+\log(d)+m\log(1/\gamma)\right).
		\end{equation*}
	\end{lemma}
	
	\begin{proof}[Proof of Lemma \protect\ref{lm: discrete error RBM}]
		By Lemma \ref{lemma: lpsz}, we have
		\begin{equation*}
		\begin{aligned} &E[\|\mathbf{Y}^m(t)-\mathbf{Y}(t)\|_\infty^2]\leq
		\frac{4\kappa_0^2}{\beta^2}E\left[\sup_{0\leq s\leq
			t}\|\mathbf{X}^m(s)-\mathbf{X}(s)\|_\infty^2\right]\\
		=~&\frac{4\kappa_0^2}{\beta^2}E\left[\max_{1\leq i\leq d}\max_{0\leq s\leq
			t}|X^m_i(s)-X_i(s)|^2\right]\leq
		C_1\gamma^m\left(\log(t)+\log(d)+m\log(1/\gamma)\right), \end{aligned}
		\end{equation*}
		the last inequality follows from Lemma \ref{lm: discrete error BM} with $%
		C_1=C_0\cdot \frac{4\kappa_0^2}{\beta^2}$.
	\end{proof}
	
	\subsection{Non-stationary Error Bound}
	
	\label{sec:non_stationary_bd} Convergence rate to stationarity of RBM has
	been analyzed in \cite{Budhiraja_2019} and \cite{BlanchetChen2017} based on
	the synchronous coupling technique. Here we provide an alternative method
	based on the derivative of RBM with respect to the initial condition.
	Intuitively, the non-stationary error should have the same order to this
	derivative, as it reflects the impact of the initial condition on the RBM.
	
	To do this, we first introduce the directional derivative of RBM as defined
	in \cite{MandelbaumRamanan_2010}. For every continuous input $\mathbf{X}%
	_{0:t}$ and any initial condition $\mathbf{y}$, \cite{MandelbaumRamanan_2010}
	shows that there exists a matrix valued process $\mathfrak{D}\left( t;%
	\mathbf{y},\mathbf{X}_{0:t}\right) $ such that
	\begin{equation*}
	\mathfrak{D}\left( t;\mathbf{y},\mathbf{X}_{0:t}\right) \cdot \mathbf{h}%
	=\lim_{\varepsilon \rightarrow 0}\frac{\mathbf{Y}\left( t;\mathbf{y+}%
		\varepsilon \mathbf{h},\mathbf{X}_{0:t}\right) -\mathbf{Y}\left( t;\mathbf{y}%
		,\mathbf{X}_{0:t}\right) }{\varepsilon }, \forall ~ \mathbf{h}\in \mathbb{R}%
	^d.
	\end{equation*}
	We first show that the derivative matrix is bounded by the product of a
	series of matrices. Following the notations introduced in Section 3 of \cite%
	{BlanchetChen2017}, for the RBM $\mathbf{Y}\left( \cdot;\mathbf{y},\mathbf{X}%
	\right)$ starting from position $\mathbf{y}$ at time $0$, define a series of
	stopping times:
	\begin{equation}  \label{eq: eta}
	\begin{aligned} \eta_{i}^{k}\left( \mathbf{y}\right) &
	=\inf\{t>\eta^{k-1}\left( \mathbf{y}\right) +1:Y_{i}\left(
	t;\mathbf{y}\right) =0\}, \\ \eta^{k}\left( \mathbf{y}\right) &
	=\sup\{\eta_{i}^{k}\left( \mathbf{y}\right) :1\leq i\leq d\}, \end{aligned}
	\end{equation}
	and set-valued functions
	\begin{align*}
	\Gamma_{i}\left( t,\mathbf{y}\right) =\{\eta_{i}^{k}:\eta_{i}^{k}\leq t\},
	\text{ and } \Gamma\left( t,\mathbf{y}\right) =\cup_{i=1}^{d}\Gamma
	_{i}\left( t,\mathbf{y}\right) .
	\end{align*}
	For any time point $t\geq0$, define%
	\begin{equation*}
	\mathcal{C}\left( t\right) =\{1\leq i\leq d:Y_{i}\left( t\right) =0\}\text{
		and }\mathcal{\bar{C}}\left( t\right) =\{1\leq j\leq d:j\notin\mathcal{C}%
	\left( t\right) \}
	\end{equation*}
	For any subset $S$ of $\{1,2,...,d\}$, define the $d\times d$ matrix $%
	\Lambda\left( S\right) $ as
	\begin{equation*}
	\Lambda_{i,j}\left( S\right) =P_{i}\left( \tau\left( S\right) <\tau\left(
	\{0\}\right) ,W\left( \tau\left( S\right) \right) =j\right)\text{ for }%
	i,j\in\{1,...,d\}.
	\end{equation*}
	The following result provides an explicit bound for the derivative matrix in
	terms of the product of $\Lambda$ matrices.
	
	\begin{lemma}
		\label{lemma: derivative} The derivative process
		\begin{equation*}
		\mathfrak{D}\left( t;\mathbf{y},\mathbf{X}_{0:t}\right) \leq
		R^{-1}\prod\limits_{s\in \Gamma(t,\mathbf{y})}\Lambda^T \left( \mathcal{\bar{%
				C}}\left( s\right) \right).
		\end{equation*}
	\end{lemma}
	
	\begin{remark}
		Under the uniformity assumptions, $\|R^{-1}\mathbf{1}\|\leq b_1$, as a
		result, for any $1\leq i, j\leq d$,
		\begin{equation*}
		\mathfrak{D}_{ij}\left( t;\mathbf{y},\mathbf{Z}_{0:t}\right)\leq
		b_1\|\prod\limits_{s\in \Gamma(t,\mathbf{y})}\Lambda^T \left( \mathcal{\bar{C%
		}}\left( s\right) \right)\|_\infty.
		\end{equation*}
	\end{remark}
	
	\begin{proof}[Proof of Lemma \protect\ref{lemma: derivative}]
		For the simplicity of notation, we shall write $\mathfrak{D}(t;\mathbf{y},%
		\mathbf{X}_{0:t})=\mathfrak{D}(t)$ and define $\gamma(t)=R^{-1}(\mathcal{D}%
		(t)-I)$, i.e., $\gamma(t)$ is the derivative of $\mathbf{L}(t)$ with respect
		to the initial value $\mathbf{y}$ (see \cite{MandelbaumRamanan_2010}).
		
		According to Theorem 1.1 of \cite{KellaRamasubramanian_2012}, the process $%
		R^{-1}(\mathbf{Y}(t,\mathbf{y}_1, \mathbf{X}_{0:t})-\mathbf{Y}(t,\mathbf{y}%
		_2, \mathbf{X}_{0:t}))$ is non-increasing in $t$, for any $\mathbf{y}_1\geq
		\mathbf{y}_2$. As a direct consequence, we can conclude that $\gamma(t)$ is
		non-increasing in $t$ component by component.
		
		Suppose $\Gamma(t,\mathbf{y})=\{\tau_1, \tau_2,...\}$ with $%
		\tau_1<\tau_2<... $ in order. We define
		\begin{equation*}
		D_n=\prod_{k\leq n}\Lambda^T(\bar{\mathcal{C}}(\tau_k)), \text{ and }%
		\gamma_n=R^{-1}(D_n-I).
		\end{equation*}
		In particular, $D_0=I$ and $\gamma_0=0$. We shall prove by induction that
		for any $\tau_{n}<t\leq\tau_{n+1}$,
		\begin{equation}  \label{Eq: gamma}
		\gamma(t)\leq \gamma_n\text{ and hence }R^{-1}\mathfrak{D}(t)\leq
		R^{-1}\prod_{k\leq N(t)}\Lambda^T(\bar{\mathcal{C}}(\tau_k)).
		\end{equation}
		
		First, when $t<\tau _{1}$, by definition $\gamma (t)=\gamma _{0}=0$. Now
		suppose \eqref{Eq: gamma} holds for all $n\leq m-1$ and we consider a fixed
		time $\tau _{m}<t\leq \tau _{m+1}$. According to \cite%
		{MandelbaumRamanan_2010}, the derivative processes $\gamma (t)$ is the
		unique solution to the following system of equations:
		\begin{equation*}
		\gamma _{ij}(t)=\sup_{s\in \Phi ^{i}(t)}\left[ -I_{ij}+(P\gamma (s))_{ij}%
		\right] ,
		\end{equation*}%
		where $\Phi ^{i}(t)=\{s\leq t:L_{i}(s)=L_{i}(t)\}$ and $P=I-R\geq 0$. For
		any $i\in \mathcal{C}(\tau _{k})$, $L_{i}(s)>L_{i}(\tau _{k})$ with
		probability 1. By the fact that $\gamma (t)$ is non-increasing in $t$ and $%
		P\geq 0$, we have
		\begin{equation*}
		\gamma _{ij}(t)\leq -I_{ij}+(P\gamma (\tau _{m}))_{ij}\leq -I_{ij}+(P\gamma
		_{m-1})_{ij},
		\end{equation*}%
		where the last inequality holds by the induction assumption. For any $i\in
		\bar{\mathcal{C}}(\tau _{m})$, we have $\gamma _{ij}(t)\leq \gamma
		_{ij}(\tau _{m})\leq \gamma _{m-1,ij}$. Suppose $\bar{\gamma}$ is the
		solution to the systems of linear equations:
		\begin{equation*}
		\bar{\gamma}_{ij}=%
		\begin{cases}
		-I_{ij}+(P\gamma _{m-1})_{ij} & \text{ if }i\in \mathcal{C}(\tau _{m}), \\
		\gamma _{m-1,ij} & \text{ if }i\in \bar{\mathcal{C}}(\tau _{m})%
		\end{cases}%
		.
		\end{equation*}%
		Then, $\gamma (t)\leq \bar{\gamma}$ component by component. For the
		simplicity of notation, we write $\mathcal{C}=\mathcal{C}(\tau _{m})$. Then,
		$\bar{\gamma}_{ij}$ can be solved explicitly as
		\begin{equation*}
		\bar{\gamma}_{\bar{\mathcal{C}}}=\gamma _{m-1,\bar{\mathcal{C}}};\bar{\gamma}%
		_{\mathcal{C}}=R_{\mathcal{CC}}^{-1}(-I_{\mathcal{C}}+P_{\mathcal{C\bar{C}}%
		}\gamma _{m-1,\mathcal{\bar{C}}}).
		\end{equation*}%
		More precisely, we write
		\begin{equation*}
		\bar{\gamma}=\left(
		\begin{array}{c}
		-R_{\mathcal{CC}}^{-1}I_{\mathcal{C}} \\
		0%
		\end{array}%
		\right) +\left(
		\begin{array}{cc}
		0 & R_{\mathcal{CC}}^{-1}P_{\mathcal{C\bar{C}}} \\
		0 & I_{\mathcal{\bar{C}\bar{C}}}%
		\end{array}%
		\right) \gamma _{m-1}.
		\end{equation*}%
		One can check that
		\begin{equation*}
		\Lambda _{m}^{T}\triangleq \Lambda ^{T}(\bar{\mathcal{C}}(\tau
		_{m}))=I+R\left(
		\begin{array}{c}
		-R_{\mathcal{CC}}^{-1}I_{\mathcal{C}} \\
		0%
		\end{array}%
		\right) ,
		\end{equation*}%
		and
		\begin{align*}
		R^{-1}\Lambda _{m}^{T}R& =R^{-1}\left[ I+R\left(
		\begin{array}{c}
		-R_{\mathcal{CC}}^{-1}I_{\mathcal{C}} \\
		0%
		\end{array}%
		\right) \right] R=I+\left(
		\begin{array}{cc}
		-R_{\mathcal{CC}}^{-1}I_{\mathcal{CC}} & 0 \\
		0 & 0%
		\end{array}%
		\right) \left(
		\begin{array}{cc}
		R_{\mathcal{CC}} & R_{\mathcal{C\bar{C}}} \\
		R_{\mathcal{\bar{C}C}} & R_{\mathcal{\bar{C}\bar{C}}}%
		\end{array}%
		\right) \\
		& =I+\left(
		\begin{array}{cc}
		-I_{\mathcal{CC}} & -R_{\mathcal{CC}}^{-1}R_{\mathcal{C\bar{C}}} \\
		0 & 0%
		\end{array}%
		\right) =\left(
		\begin{array}{cc}
		0 & R_{\mathcal{CC}}^{-1}P_{\mathcal{C\bar{C}}} \\
		0 & I_{\mathcal{\bar{C}\bar{C}}}%
		\end{array}%
		\right) .
		\end{align*}%
		Therefore, we have
		\begin{eqnarray*}
			\bar{\gamma} &=&R^{-1}(\Lambda _{m}^{T}-I)+R^{-1}\Lambda _{m}^{T}R\gamma
			_{m-1} \\
			&=&R^{-1}(\Lambda _{m}^{T}-I)+R^{-1}\Lambda _{m}^{T}R\cdot
			R^{-1}(\prod_{k\leq m-1}\Lambda _{k}^{T}-I) \\
			&=&R^{-1}(\prod_{k\leq m}\Lambda _{k}^{T}-I)=\gamma _{m}.
		\end{eqnarray*}%
		As a result, \eqref{Eq: gamma} holds by induction and we have
		\begin{equation*}
		R^{-1}\mathfrak{D}(t)\leq R^{-1}\prod_{k\leq N(t)}\Lambda ^{T}(\bar{\mathcal{%
				C}}(\tau _{k})).
		\end{equation*}%
		Since all the components of $R^{-1}$ are nonnegative and all its diagonal
		entries are greater or equal to 1, we can conclude that, component by
		component
		\begin{equation*}
		\mathfrak{D}(t)\leq R^{-1}\prod_{k\leq N(t)}\Lambda ^{T}(\bar{\mathcal{C}}%
		(\tau _{k})).
		\end{equation*}
	\end{proof}
	
	Now, we are ready to derive the upper bound for non-stationary error.
	
	\begin{lemma}
		\label{lemma: stationary error} There exists constants $C_2$ and $\xi_1>0$
		such that
		\begin{equation*}
		E[\|\mathbf{Y}(t;\mathbf{Y}(\infty), \mathbf{X}_{0:t})-\mathbf{Y}(t;0,
		\mathbf{X}_{0:t})\|_\infty^2] \leq C_2d^3 \exp\left(-\xi_1\frac{t}{\log(d)}
		\right).
		\end{equation*}
	\end{lemma}
	
	\begin{proof}[Proof of Lemma \protect\ref{lemma: stationary error}]
		By the definition of directional derivative of RBM, for any $\mathbf{y}\in%
		\mathbb{R}^d_+$,
		\begin{align*}
		\mathbf{Y}\left( t;\mathbf{y},\mathbf{X}_{0:t}\right) - \mathbf{Y}\left( t;0,%
		\mathbf{X}_{0:t}\right) = \left(\int_{0}^{1} \mathfrak{D}\left( t;u\cdot%
		\mathbf{y},\mathbf{X}_{0:t}\right) du\right) \mathbf{y}.
		\end{align*}
		Then, following Lemma \ref{lemma: derivative}, for $j=1,2,...,d$,
		\begin{equation*}
		|Y_j(t;\mathbf{y},\mathbf{X}_{0:t})-Y_j(t;0,\mathbf{X}_{0:t})|\leq%
		\sum_{i=1}^d b_1y_i \int_0^{1}\|\prod\limits_{s\in \Gamma(t,u\cdot \mathbf{y}%
			)}\Lambda^T \left( \mathcal{\bar{C}}\left( s\right) \right)\|_\infty du .
		\end{equation*}
		Therefore,
		\begin{equation*}
		\|\mathbf{Y}(t;\mathbf{y},\mathbf{X}_{0:t})-\mathbf{Y}(t;0,\mathbf{X}%
		_{0:t})\|_\infty\leq b_1 \int_0^{1}\|\prod\limits_{s\in \Gamma(t,u\cdot
			\mathbf{y})}\Lambda^T \left( \mathcal{\bar{C}}\left( s\right)
		\right)\|_\infty du\cdot \|\mathbf{y}\|_1.
		\end{equation*}
		Let's denote $\|\prod\limits_{s\in \Gamma(t,u\cdot \mathbf{y}/y_i)}\Lambda^T
		\left( \mathcal{\bar{C}}\left( s\right) \right)\|_\infty=\Theta(u)$. Then we
		have
		\begin{equation*}
		\|\mathbf{Y}(t;\mathbf{y},\mathbf{X}_{0:t})-\mathbf{Y}(t;0,\mathbf{X}%
		_{0:t})\|^2_\infty\leq b_1^2\|\mathbf{y}\|_1^2\left(\int_0^1
		\Theta(u)du\right)^2 \leq b_1^2\|\mathbf{y}\|_1^2\int_0^1 \Theta(u)du.
		\end{equation*}
		The last equality holds as $\Theta(u)\leq 1$ for all $0\leq u\leq 1$.
		
		The rest of proof follows the same argument as in \cite{Budhiraja_2019}. By
		Lemma 2 and Lemma 3 of \cite{BlanchetChen2017}, all $0\leq u\leq 1$,
		\begin{equation*}
		\Theta(u)\leq\| Q^{\mathcal{N} (t,u\cdot\mathbf{y})}\mathbf{1}\|_\infty ,
		\end{equation*}
		where $\mathcal{N}(t,\mathbf{y})$ is a random positive integer equals
		basically to the number of stopping times $\eta_k$ (defined by
		\eqref{eq:
			eta}) observed by time $t$, (see also \cite{BlanchetChen2017} for the
		details). Then, we have
		\begin{equation*}
		\begin{aligned} & E[\|\mathbf{Y}(t;\mathbf{Y}(\infty),
		\mathbf{X}_{0:t})-\mathbf{Y}(t;0,\mathbf{X}_{0:t})\|_\infty^2] \leq b_1^2
		E\left[\|\mathbf{Y}(\infty)\|_1^2\|Q^{\mathcal{N}(t,\mathbf{Y}(\infty)}
		\mathbf{1}\|_\infty\right]\\ \leq ~&
		b_1^2E[\|\mathbf{Y}(\infty)\|_1^4]^{1/2}E[\|Q^{\mathcal{N}(t,\mathbf{Y}(%
			\infty))} \mathbf{1}\|_\infty^2]^{1/2}\leq
		b_1^2E[\|\mathbf{Y}(\infty)\|_1^4]^{1/2}E[\|Q^{\mathcal{N}(t,\mathbf{Y}(%
			\infty))} \mathbf{1}\|_\infty]^{1/2}. \end{aligned}
		\end{equation*}
		The proof of Theorem 1 of \cite{Budhiraja_2019} (page 20) shows that, under
		Assumptions A1) to A3),
		\begin{equation*}
		\begin{aligned} E[\|\mathbf{Y}(\infty)\|_1^4]^{1/2} ~&\leq
		\frac{4b^2_0}{\delta_0^2}d^2,\\
		E\left[\|Q^{\mathcal{N}(t,\mathbf{Y}(\infty)}\mathbf{1}\|_\infty%
		\right]^{1/2}~& \leq
		C_0d\left(\exp\left(-\xi_1\frac{t}{\log(d)}\right)\right). \end{aligned}
		\end{equation*}
		Therefore, let $C_2= \frac{4b^2_0}{\delta_0^2}C_0$, we get
		\begin{equation*}
		E[\|\mathbf{Y}(t;\mathbf{Y}(\infty), \mathbf{X}_{0:t})-\mathbf{Y}(t;0,
		\mathbf{X}_{0:t})\|_1^2] \leq C_2 d^3 \exp\left(-\xi_1\frac{t}{\log(d)}
		\right).
		\end{equation*}
	\end{proof}

	\subsection{Complexity Analysis}
	
	\label{sec: complexity} Given the error bounds Lemma \ref{lm: discrete error
		BM} and Lemma \ref{lemma: stationary error}, we are ready to show that, for
	the two-parameter multilevel Monte Carlo Algorithm \ref{alg: 1}, the
	computational budget to obtain estimator of a fixed accuracy level is
	almost linear in dimension $d$.
	
	\begin{proof}[Proof of Theorem \protect\ref{thm: 1}]
		Recall that for a given sequence of RBMs and the $f$ function to evaluate,
		the algorithm has five input parameters $(\gamma, T, L,\mathbf{y}_0,N)$. In
		the following analysis, we choose $\mathbf{y}_0 = 0$.
		
		For fixed $d$ and $\varepsilon $, the mean square error of the estimator $%
		\bar{Z}$ can be expressed as
		\begin{eqnarray}
		&&E[(\bar{Z}-E[f(\mathbf{Y}(\infty ))])^{2}]  \notag \\
		&=&Var\left[ \bar{Z}\right] +(E[\bar{Z}]-E[f(\mathbf{Y}(\infty ))])^{2}
		\notag \\
		&\leq &\frac{1}{N}E[(Z-f(\mathbf{y}_{0}))^{2}]+(E[Z]-E\left[ f(\mathbf{Y}%
		(\infty ))\right] )^{2}  \notag \\
		&=&\frac{1}{N}\sum_{m=0}^{L-1}p(m)^{-1}E\left[ \left( f\left( \mathbf{Y}%
		^{m+1}\left( (m+1)T;\mathbf{y}_{0},\mathbf{X}_{0:(m+1)T}\right) \right)
		-f\left( \mathbf{Y}^{m}\left( mT;\mathbf{y}_{0},\mathbf{X}_{T:(m+1)T}\right)
		\right) \right) ^{2}\right]  \notag \\
		&&+(E[Z]-E[f(\mathbf{Y}(\infty ))])^{2}  \notag \\
		&\triangleq &\frac{1}{N}\sum_{m=0}^{L-1}K(\gamma )^{-1}\gamma
		^{-m}V_{m}+Bias^{2}.  \label{eq: MSE2}
		\end{eqnarray}
		
		We first analyze the variance terms $V_{m}$ for each $m=0,1,...,L-1$.
		Following Assumption A4),
		\begin{eqnarray*}
			&&f\left( \mathbf{Y}^{m+1}\left( (m+1)T;\mathbf{y}_{0},\mathbf{X}%
			_{0:(m+1)T}\right) \right) -f\left( \mathbf{Y}^{m}\left( mT;\mathbf{y}_{0},%
			\mathbf{X}_{T:(m+1)T}\right) \right) \\
			&\leq &\mathcal{L}\Vert \mathbf{Y}^{m+1}\left( (m+1)T;\mathbf{y}_{0},\mathbf{%
				X}_{0:(m+1)T}\right) -\mathbf{Y}^{m}\left( mT;\mathbf{y}_{0},\mathbf{X}%
			_{T:(m+1)T}\right) \Vert _{\infty },
		\end{eqnarray*}%
		and
		\begin{eqnarray*}
			&&\Vert \mathbf{Y}^{m+1}\left( (m+1)T;\mathbf{y}_{0},\mathbf{X}%
			_{0:(m+1)T}\right) -\mathbf{Y}^{m}\left( mT;\mathbf{y}_{0},\mathbf{X}%
			_{T:(m+1)T}\right) \Vert _{\infty } \\
			&\leq &\Vert \mathbf{Y}^{m+1}\left( (m+1)T;\mathbf{y}_{0},\mathbf{X}%
			_{0:(m+1)T}\right) -\mathbf{Y}\left( (m+1)T;\mathbf{y}_{0},\mathbf{X}%
			_{0:(m+1)T}\right) \Vert _{\infty } \\
			&&+\Vert \mathbf{Y}^{m}\left( mT;\mathbf{y}_{0},\mathbf{X}_{T:(m+1)T}\right)
			-\mathbf{Y}\left( mT;\mathbf{y}_{0},\mathbf{X}_{T:(m+1)T}\right) \Vert
			_{\infty } \\
			&&+\Vert \mathbf{Y}\left( (m+1)T;\mathbf{y}_{0},\mathbf{X}_{0:(m+1)T}\right)
			-\mathbf{Y}\left( mT;\mathbf{y}_{0},\mathbf{X}_{T:(m+1)T}\right) \Vert
			_{\infty }
		\end{eqnarray*}
		
		Following Lemma \ref{lm: discrete error RBM}, we have
		\begin{align*}
		& E[\Vert \mathbf{Y}^{m+1}\left( (m+1)T;\mathbf{y}_{0},\mathbf{X}%
		_{0:(m+1)T}\right) -\mathbf{Y}\left( (m+1)T;\mathbf{y}_{0},\mathbf{X}%
		_{0:(m+1)T}\right) \Vert _{\infty }^{2})] \\
		& \leq C_{1}\gamma ^{m+1}\left( \log (\left( m+1\right) T)+\log (d)+\left(
		m+1\right) \log (1/\gamma )\right) ,
		\end{align*}%
		and
		\begin{eqnarray*}
			&&E[\Vert \mathbf{Y}^{m}\left( mT;\mathbf{y}_{0},\mathbf{X}%
			_{T:(m+1)T}\right) -\mathbf{Y}\left( mT;\mathbf{y}_{0},\mathbf{X}%
			_{T:(m+1)T}\right) \Vert _{\infty }^{2}] \\
			&\leq &C_{1}\gamma ^{m}\left( \log (mT)+\log (d)+m\log (1/\gamma )\right) .
		\end{eqnarray*}%
		Following Lemma \ref{lemma: stationary error}, we have
		\begin{eqnarray*}
			&&E\left[ \Vert \mathbf{Y}\left( (m+1)T;\mathbf{y}_{0},\mathbf{X}%
			_{0:(m+1)T}\right) -\mathbf{Y}\left( mT;\mathbf{y}_{0},\mathbf{X}%
			_{T:(m+1)T}\right) \Vert _{\infty }^{2}\right] \\
			&=&E[\Vert \mathbf{Y}\left( mT;\mathbf{Y}(T;\mathbf{y}_{0},\mathbf{X}_{0:T}),%
			\mathbf{X}_{0:mT}\right) -\mathbf{Y}\left( mT;\mathbf{y}_{0},\mathbf{X}%
			_{0:mT}\right) \Vert _{\infty }^{2}] \\
			&\leq &C_{2}\cdot d^{3}\exp \left( -\xi _{1}\frac{mT}{\log d}\right) .
		\end{eqnarray*}%
		Therefore, recall that $(a+b+c)^2 \leq 3\left(a^2 + b^2 + c^2\right)$, we
		have
		\begin{equation*}
		V_{m}\leq 3\mathcal{L}^{2}\left( 2C_{1}\gamma ^{m}(\log ((m+1)T)+\log
		(d)+(m+1)\log (1/\gamma ))+C_{2}\cdot d^{3}\exp \left( -\xi _{1}\frac{mT}{%
			\log d}\right) \right) .
		\end{equation*}%
		Let $T=\lceil \left( 3\log (d)^{2}+\log (1/\gamma )\log (d)\right) /\xi
		_{1}\rceil $ and $C_{3}=3\mathcal{L}^{2}(2C_{1}+C_{2})$. We have
		\begin{eqnarray*}
			V_{m} &\leq &3\mathcal{L}^{2}\left( 2C_{1}\gamma ^{m}(\log ((m+1)T)+\log
			(d)+(m+1)\log (1/\gamma ))+C_{2}\gamma ^{m}\right) \\
			&\leq &C_{3}\gamma ^{m}(\log ((m+1)T)+\log (d)+(m+1)\log (1/\gamma )).
		\end{eqnarray*}%
		Therefore, the total variance of our estimator is
		\begin{align*}
		V_{total}=~& \frac{1}{N}\sum_{m=0}^{L-1}K(\gamma )^{-1}\gamma ^{-m}V_{m} \\
		\leq ~& \frac{1}{N}K(\gamma )^{-1}\sum_{m=0}^{L-1}C_{3}(\log ((m+1)T)+\log
		(d)+(m+1)\log (1/\gamma ))) \\
		\leq ~& \frac{1}{N}C_{3}K(\gamma )^{-1}L\left( \log (LT)+\log (d)+L\log
		(1/\gamma )\right) \leq \varepsilon ^{2}/2 .
		\end{align*}
		
		Now we turn to the term of bias in \eqref{eq: MSE2}. Following Assumption
		A4), we have
		\begin{align*}
		& Bias^{2} \\
		=~& (E[Z]-E[f(\mathbf{Y}(\infty ))])^{2}=(E[f(\mathbf{Y}^{L}(TL;\mathbf{y}%
		_{0},\mathbf{X}_{0:LT}))-f(\mathbf{Y}(\infty ))])^{2} \\
		\leq ~& 2\left( (E[f(\mathbf{Y}^{L}(TL;\mathbf{y}_{0},\mathbf{X}_{0:LT}))-f(%
		\mathbf{Y}(TL;\mathbf{y}_{0},\mathbf{X}_{0:LT}))])^{2}+(E[f(\mathbf{Y}(TL;%
		\mathbf{y}_{0},\mathbf{X}_{0:LT}))-f(\mathbf{Y}(\infty ))])^{2}\right) \\
		\leq ~& 2\left( E\left[ \left( f(\mathbf{Y}^{L}(TL;\mathbf{y}_{0},\mathbf{X}%
		_{0:LT}))-f(\mathbf{Y}(TL;\mathbf{y}_{0},\mathbf{X}_{0:LT}))\right) ^{2}%
		\right] +E\left[ \left( f(\mathbf{Y}(TL;\mathbf{y}_{0},\mathbf{X}_{0:LT}))-f(%
		\mathbf{Y}(\infty ))\right) ^{2}\right] \right) \\
		\leq ~& 2\mathcal{L}^{2}\left( E[\Vert \mathbf{Y}^{L}(TL;\mathbf{y}_{0},%
		\mathbf{X}_{0:LT})-\mathbf{Y}(TL;\mathbf{y}_{0},\mathbf{X}_{0:LT})\Vert
		_{\infty }^{2}]+E[\Vert \mathbf{Y}(TL;\mathbf{y}_{0},\mathbf{X}_{0:LT})-%
		\mathbf{Y}(\infty )\Vert _{\infty }^{2}]\right) .
		\end{align*}%
		Following Lemma \ref{lm: discrete error RBM}, we have
		\begin{equation*}
		E[\Vert \mathbf{Y}^{L}(TL;\mathbf{y}_{0},\mathbf{X}_{0:LT})-\mathbf{Y}(TL;%
		\mathbf{y}_{0},\mathbf{X}_{0:LT})\Vert _{\infty }^{2}]\leq C_{1}\left(
		\gamma ^{L}(\log (LT)+\log (d)+L\log (1/\gamma ))\right) .
		\end{equation*}%
		Following Lemma \ref{lemma: stationary error}, we have
		\begin{equation*}
		E[\Vert \mathbf{Y}(TL;\mathbf{y}_{0},\mathbf{X}_{0:LT})-\mathbf{Y}(\infty
		)\Vert _{\infty }^{2}]\leq C_{2}\cdot d^{3}\exp \left( -\xi _{1}\frac{LT}{%
			\log (d)}\right) \leq C_{2}\cdot \gamma ^{L},
		\end{equation*}%
		for $T=\lceil \left( 3\log (d)^{2}+\log (1/\gamma )\log (d)\right) /\xi
		_{1}\rceil .$
		
		Therefore,
		\begin{equation*}
		Bias^{2}\leq C_{3}\left( \gamma ^{L}(\log (LT)+\log (d)+L\log (1/\gamma
		))\right) \leq \varepsilon ^{2}/2,
		\end{equation*}%
		for $T=\lceil \left( 3\log (d)^{2}+\log (1/\gamma )\log (d)\right) /\xi
		_{1}\rceil $ and $L= \lceil \left( \log (\log (d))+2\log (1/\varepsilon
		)+k_{1}\right) /\log (1/\gamma ) \rceil ,$ where $k_{1}$ is a numerical
		constant.
		
		To equalize the variance and bias of our estimator, we enforce%
		\begin{equation*}
		C_{3}\left( \gamma ^{L}(\log (LT)+\log (d)+L\log (1/\gamma ))\right) =~\frac{%
			1}{N}C_{3}K(\gamma )^{-1}L\left( \log (LT)+\log (d)+L\log (1/\gamma )\right)
		.
		\end{equation*}%
		So, $N=K(\gamma )^{-1}L/\gamma ^{L}=O\left( \varepsilon ^{-2}K(\gamma
		)^{-1}L\log (d)\right) .$
		
		Note that the complexity, in terms of expected random seeds used, to
		simulate one sample of $Z$, should be
		\begin{equation*}
		\mathcal{C}=\sum_{m=0}^{L-1}p(m)\gamma ^{-(m+1)}T(m+1)d=\frac{1}{2}K(\gamma
		)\gamma ^{-1}dTL(L+1).
		\end{equation*}%
		Then, the total complexity to compute $\bar{Z}$ by $N$ rounds of simulation,
		with our choice of $(\gamma ,T,L,N)$, is
		\begin{align}
		N\times \mathcal{C}& =O\left( \varepsilon ^{-2}K(\gamma )^{-1}L\log
		(d)\right) \times \left( \frac{1}{2}K(\gamma )\gamma ^{-1}dTL(L+1)\right)
		\notag \\
		=~& O\left( \varepsilon ^{-2}dT\log (d)L^{3}\right) =~O\left( \varepsilon
		^{-2}d\log (d)^{3}(\log (\log (d))+\log (1/\varepsilon ))^{3}\right) .
		\label{eqn:tot_comp}
		\end{align}
	\end{proof}
	
	\begin{lemma}
		\label{lmm: gamma} The optimal $\gamma^* = 0.05$.
	\end{lemma}
	
	\begin{proof}[Proof of Lemma \protect\ref{lmm: gamma}]
		According to (\ref{eqn:tot_comp}), we have the dependence of the total
		complexity on $\gamma $ is approximately $\gamma ^{-1}\left( \log (1/\gamma
		)\right) ^{-3}.$ We shall optimize $\gamma $ to obtain the optimal
		complexity. Therefore, the optimal $\gamma $ is
		\begin{equation*}
		\gamma ^{\ast }=\arg \min_{0<\gamma <1}\gamma ^{-1}\left( \log (1/\gamma
		)\right) ^{-3}=0.05.
		\end{equation*}
	\end{proof}
	
	\section{Conclusion}
	
	We have presented and analyzed a Monte Carlo strategy which provides
	asymptotically optimal estimators for steady-state expectations of
	high-dimensional RBM. We believe that the strategy that we present can be
	applied to more general networks. A key idea is to consider the so-called
	asynchronous coupling in combination of multilevel Monte Carlo. While this
	idea is not new (see, for example, \cite{Glynn_Rhee_14}), the analysis,
	which is based on the rate of decay to zero of the product of sub-stochastic
	random matrices is, we believe, applicable to other settings. In particular,
	the sensitivity to the initial condition in every stochastic flow naturally
	yields to the study of product of random matrices and the analysis of the
	so-called top-Lyapunov exponent. In this paper, we are able to use implicit
	estimates for this product from \cite{Budhiraja_2019} and \cite%
	{BlanchetChen2017}. This, we expect, will provide a blueprint that can be
	used in other settings, as we expect to report in future research.\\
	
	\textbf{Acknowledgement:} J. Blanchet gratefully acknowledges NSF grants No. 1915967, 1820942, 1838576. X. Chen gratefully acknowledges NSFC grants No. 91646206 and 11901493.
	\bibliographystyle{plain}
	\bibliography{references}

\begin{thebibliography}{10}

\bibitem{Budhiraja_2019}
S.~Banerjee and A.~Budhiraja.
\newblock Parameter and dimension dependence of convergence rates to
  stationarity for reflecting {B}rownian motions.
\newblock {\em Working paper}, 2019.

\bibitem{BlanchetChen2017}
J.~Blanchet and X.~Chen.
\newblock Rates of convergence to stationarity for multidimensional {RBM}.
\newblock {\em Mathematics of Operations Research, preprint}, 2020.

\bibitem{cottle1992linear}
R.~W. Cottle, J.~S. Pang, and R.~E. Stone.
\newblock {\em The linear complementarity problem}.
\newblock 1992.

\bibitem{DaiHarrison_1992}
J.~G. Dai and J.~M. Harrison.
\newblock Reflected {B}rownian motion in an orthant: Numerical methods for
  steady-state analysis.
\newblock {\em Annals of Applied Probability}, 2:65--86, 1992.

\bibitem{Giles_2008}
M.~B. Giles.
\newblock Multilevel monte carlo path simulation.
\newblock {\em Operations Research}, 56:607--617, 2008.

\bibitem{Giles_15}
M.~B. Giles.
\newblock Multilevel monte carlo methods.
\newblock {\em Acta Numerica}, 24:259--328, 2015.

\bibitem{GMSVZ_19}
M.~B. Giles, M.~Majka, L.~Szpruch, S.~Vollmer, and K.~Zygalakis.
\newblock Multi-level monte carlo methods for the approximation of invariant
  measures of stochastic differential equations.
\newblock {\em Statistics and Computing}, 2019.

\bibitem{Glynn_Rhee_14}
P.~W. Glynn and C.~H. Rhee.
\newblock Exact estimation for markov chain equilibrium expectations.
\newblock {\em Journal of Applied Probability}, 51A:377--389, 2014.

\bibitem{HarrisonReiman_1981}
J.~M. Harrison and M.~I. Reiman.
\newblock Reflected {B}rownian motion on an orthant.
\newblock {\em Annals of Applied Probability}, 9:302--308, 1981.

\bibitem{HarrisonWilliams_1987b}
J.~M. Harrison and R.~J. Williams.
\newblock Brownian models of open queueing networks with homogeneous customer
  populations.
\newblock {\em Stochastics}, 2:77--115, 1987.

\bibitem{karlin1981second}
S.~Karlin and H.~E. Taylor.
\newblock {\em A second course in stochastic processes}.
\newblock Elsevier, 1981.

\bibitem{KellaRamasubramanian_2012}
O.~Kella and S.~Ramasubramanian.
\newblock Asymptotic irrelevance of initial conditions for {S}korokhod
  refection mapping on the nonnegative orthant.
\newblock {\em Mathematics of Operations Research}, 37:301--312, 2012.

\bibitem{MandelbaumRamanan_2010}
A.~Mandelbaum and K.~Ramanan.
\newblock Directional derivatives of oblique reflection maps.
\newblock {\em Mathematics of Operations Research}, 35:527--558, 2010.

\bibitem{Glynn_Rhee_15}
C.~H. Rhee and P.~W. Glynn.
\newblock Unbiased estimation with square root convergence for sde models.
\newblock {\em Operations Research}, 63:1026--1043, 2015.

\end{thebibliography}
	\clearpage
	\appendix
	\section{Routine to solve Skorokhod Problem in Algorithm \ref{alg: 1}}
	In Step 6 of Algorithm \ref{alg: 2}, once the piece-wise linear approximation is obtained for the
	underlying Brownian motion, we obtain the solution to the Skorokhod problem
	by solving, at each time-step, a static linear complementarity problem (see, for example \cite{cottle1992linear}). Since $R$ is an $M$-matrix, we here provide a simple yet numerical stable algorithm to solve the linear complementarity problem in Algorithm \ref{alg: 2}.
	\begin{algorithm}[H]
		\caption{Algorithm for the Linear Complementarity Problem}
		\label{alg: 2}
		\begin{flushleft}
			\hspace*{0.02in} {\bf Input:}\\
			The reflection matrix: $R$;\\
			The initial vector: $\mathbf{x}$;\\
			\hspace*{0.02in} {\bf Output:}\\
			The solution of the linear complementarity problem: $\mathbf{y}\geq \mathbf{0}$, where $\mathbf{y}=\mathbf{x}+R\mathbf{L}$ for $\mathbf{L}\geq \mathbf{0}$. \\
		\end{flushleft}
		
		\begin{algorithmic}[1]
			\State Set $\epsilon = 10^{-8}$;
			\State $\mathbf{y}=\mathbf{x}$;
			\While{Exists $\mathbf{y}_i<-\epsilon$}
			\State Compute the set $B = \{i: \mathbf{y}_i<\epsilon$\};
			\State Compute $\mathbf{L}_B=-R_{B,B}^{-1} \mathbf{x}_B$;
			\State Compute $\mathbf{y} = \mathbf{x} + R_{:,B} \times \mathbf{L}_B$;

			\EndWhile
			\Return $\mathbf{y}$.
		\end{algorithmic}
	\end{algorithm}
	\section{Lower Bound on Constant $\xi_1$}
	We also provide a lower bound for the constant $\xi _{1}$, which is not
	given explicitly in either \cite{Budhiraja_2019} or \cite{BlanchetChen2017}.
	The lower bound is computed based on a worst-case analysis in \cite%
	{Budhiraja_2019}. We believe that it is far from tight, as shown in the
	numerical experiments in Section \ref{sec: numerics}. We provide this,
	nevertheless, for completeness.
	
	\begin{lemma}
		\label{lmm: xi} The constant $\xi_1$ satisfies
		\begin{equation*}
		\xi_1\geq D_1\left(\frac{\log(2)}{\log(1-\beta_0)^{-1}}+1\right)^{-1}\left(2+%
		\frac{\kappa_0^2b_0}{\beta_0^2\delta_0^2}\right)^{-1}
		\end{equation*}
		with $D_1 = 1/557065$.
	\end{lemma}
	
	\begin{proof}[Proof of Lemma \protect\ref{lmm: xi}]
		Our $\xi_1$ is equivalent to $E_2$ as defined in Theorem 3 in \cite%
		{Budhiraja_2019}, i.e.
		\begin{equation*}
		E_2 = D_1\left(\frac{\log(2)}{\log(1-\beta_0)^{-1}}+1\right)^{-1}\left(2+%
		\frac{\kappa_0^2b_0}{\beta_0^2\delta_0^2}\right)^{-1},
		\end{equation*}
		with $D_1 = \delta^{\prime }/128$, $\delta^{\prime }= (64C_1)^{-1}$ and $%
		C_1=C_0=A_0=68$ according to Lemma 7 and Lemma 8 in \cite{Budhiraja_2019}.
	\end{proof}
	
\end{document}